\documentclass[oneside,english,reqno]{amsart}
\usepackage[T1]{fontenc}
\usepackage[latin9]{inputenc}
\usepackage{amsthm}
\usepackage{amssymb}
\usepackage{amsmath}
\usepackage{mathtools}
\usepackage{stackrel}
\usepackage{setspace}
\usepackage{mathdots}
\usepackage{microtype}
\usepackage[foot]{amsaddr}
\usepackage[pdftex, % sets up hyperref to use pdftex driver
plainpages=false, % allows page i and 1 to exist in the same document
breaklinks=true, % link texts can be broken at the end of line
colorlinks=true,
pdftitle=My Document
pdfauthor=My Good Self
colorlinks=true,
urlcolor=blue,
citecolor=red,
linkcolor=blue
]{hyperref}

\newtheorem{thm}{\protect\theoremname}
\theoremstyle{plain}
\newtheorem*{thm*}{Theorem}
\newtheorem*{cor*}{Corollary}
\newtheorem*{not*}{Notation}
\newtheorem{prop}[thm]{\protect\propositionname}
\theoremstyle{definition}

\newtheorem{defn}[thm]{\protect\definitionname}
\theoremstyle{lem}
\newtheorem{lem}[thm]{\protect\lemmaname}
\theoremstyle{cor}
\newtheorem{cor}[thm]{\protect\corollaryname}
\theoremstyle{remark}
\newtheorem{rem}[thm]{\protect\remarkname}
\newtheorem{example}[thm]{\protect\examplename}
\theoremstyle{claim}

\makeatletter
\numberwithin{equation}{subsection}
\numberwithin{figure}{subsection}
\numberwithin{thm}{subsection}
\theoremstyle{plain}

\makeatother

\usepackage{babel}
\providecommand{\lemmaname}{Lemma}
\providecommand{\definitionname}{Definition}
\providecommand{\examplename}{Example}
\providecommand{\propositionname}{Proposition}
\providecommand{\remarkname}{Remark}
\providecommand{\theoremname}{Theorem}
\providecommand{\corollaryname}{Corollary}

\begin{document}
\onehalfspace

\author{Vincenzo Cambò$^{\dagger}$} 
\address{$\left(\dagger\right)$ SISSA, Via Bonomea 265 - 34136, Trieste, ITALY}
\email{vcambo@sissa.it}
\title{On the $E$-polynomial of parabolic $\mathrm{Sp}_{2n}$-character
varieties}

\begin{abstract}
{We find the $E$-polynomials of a family of parabolic $\mathrm{Sp}_{2n}$-character varieties $\mathcal{M}_{n}^{\xi}$
of Riemann surfaces by constructing a stratification, proving that each stratum has polynomial count, applying a result of Katz regarding the counting functions, and finally adding up the resulting $E$-polynomials of the strata. To count the number of $\mathbb{F}_{q}$-points of the strata, we invoke a formula due to Frobenius. Our calculation make use of a formula for the evaluation of characters on semisimple elements coming from Deligne-Lusztig theory, applied to the character theory of $\mathrm{Sp}{\left(2n,\mathbb{F}_{q}\right)}$, and M\"obius inversion on the poset of set-partitions. We compute the Euler characteristic of the $\mathcal{M}_{n}^{\xi}$ with these polynomials, and show they are connected.}

\end{abstract}

\maketitle
\tableofcontents
\section{Introduction}
Let $\Sigma_{g}$ be a compact Riemann surface of genus $g\geq0$
and let $G$ be a complex reductive group. The $G$-character variety
of $\Sigma_{g}$ is defined as the moduli space of representations
of $\pi_{1}{\left(\Sigma_{g}\right)}$ into $G$. Using the standard
presentation of $\pi_{1}{\left(\Sigma_{g}\right)}$, we have the following
description of this moduli space as an affine GIT quotient:
\[
\mathcal{M}_{B}{\left(G\right)}=\left\{ \left(A_{1},B_{1},\ldots,A_{g},B_{g}\right)\in G^{2g}\mid\prod\limits_{\substack{i=1}}^{g}{\left[A_{i}:B_{i}\right]}=\mathrm{Id}_{G}\right\} //G
\]
where $\left[A:B\right]:=ABA^{-1}B^{-1}$ and $G$ acts by simultaneously
conjugation.
For complex linear groups $G=\mathrm{GL}{\left(n,\mathbb{C}\right)},\mathrm{SL}{\left(n,\mathbb{C}\right)}$,
the representations of $\pi_{1}{\left(\Sigma_{g}\right)}$ into $G$
can be understood as $G$-local systems $E\rightarrow\Sigma_{g}$,
hence defining a flat bundle $E$ whose degree is zero. 

For $G=\mathrm{GL}{\left(n,\mathbb{C}\right)},\mathrm{SL}{\left(n,\mathbb{C}\right)}$, a natural generalization consists of allowing bundles $E$ of non-zero
degree $d$. In this case, one considers the space of the irreducible
$G$-local systems on $\Sigma_{g}$ with prescribed cyclic holonomy
around one puncture, which correspond to representations $\rho:\pi_{1}{\left(\Sigma_{g}\setminus\left\{p_{0}\right\} \right)}\rightarrow G$,
where $p_{0}\in\Sigma_{g}$ is a fixed point, and $\rho{\left(\gamma\right)}=e^{\frac{2\pi id}{n}}\mathrm{Id}_{G}$,
with $\gamma$ a loop around $p_{0}$, giving rise to the moduli space
\[
\mathcal{M}_{B}^{d}{\left(G\right)}=\left\{ \left(A_{1},B_{1},\ldots,A_{g},B_{g}\right)\in G^{2g}\mid\prod\limits_{\substack{i=1}}^{g}{\left[A_{i}:B_{i}\right]}=e^{\frac{2\pi id}{n}}\mathrm{Id}_{G}\right\} //G.
\]
The space $\mathcal{M}_{B}^{d}{\left(G\right)}$ is known in the literature
as the \emph{Betti moduli space}. These varieties have a very rich
structure and they have been the object of study in a broad range
of areas.

In his seminal work \cite{key-HI}, after studying the dimensional reduction
of the Yang-Mills equations from four to two dimensions, Hitchin introduced
a family of completely integrable Hamiltonian systems. These equations
are known as \emph{Hitchin's self-duality equations} on a rank $n$
and degree $d$ bundle on the Riemann surface $\Sigma_{g}$.

The moduli space of solutions comes equipped with a hyperkähler manifold
structure on its smooth locus. This hyperkähler structure has two
distinguished complex structures. One is analytically isomorphic to
$\mathcal{M}_{Dol}^{d}{\left(G\right)}$, a moduli space of $G$-Higgs
bundles, and the other is $\mathcal{M}_{DR}^{d}{\left(G\right)}$, the
space of algebraic flat bundles on $\Sigma_{g}$ of degree $d$ and
rank $n$, whose algebraic connections on $\Sigma_{g}\setminus\left\{p_{0}\right\}$
have a logarithmic pole at $p_{0}$ with residue $e^{\frac{2\pi id}{n}}\mathrm{Id}$.
By Riemann-Hilbert correspondence (\cite{key-DE}, \cite{key-SI}), the space $\mathcal{M}_{DR}^{d}{\left(G\right)}$
is analytically (but not algebraically) isomorphic to $\mathcal{M}_{B}^{d}{\left(G\right)}$
and the theory of harmonic bundles (\cite{key-CO}, \cite{key-si}) gives an homeomorphism $\mathcal{M}_{B}^{d}{\left(G\right)}\cong\mathcal{M}_{Dol}^{d}{\left(G\right)}$.

When $\mathrm{gcd}{\left(d,n\right)}=1$, these moduli spaces are smooth
and their cohomology has been computed in several particular cases,
but mostly from the point of view of the Dolbeaut moduli space $\mathcal{M}_{Dol}^{d}{\left(G\right)}$.

Hitchin and Gothen computed the Poincaré polynomial for $G=\mathrm{SL}{\left(2,\mathbb{C}\right)},\,\mathrm{SL}{\left(3,\mathbb{C}\right)}$
respectively in \cite{key-HI} and \cite{key-GO} and their techniques have been improved
to compute the compactly supported Hodge polynomials (\cite{key-GHS}).  Recently, Schiffmann, Mozgovoy and Mellit computed the Betti numbers of $\mathcal{M}_{Dol}^{d}{\left(G\right)}$ for $G=\mathrm{GL}{\left(n,\mathbb{C}\right)}$ respectively in \cite{key-Sch}, \cite{key-MS} and \cite{key-Mel}.

Hausel and Thaddeus (\cite{key-HT}) gave a new perspective for the topological
study of these varieties giving the first non-trivial example of the
Strominger-Yau-Zaslow Mirror Symmetry (\cite{key-SYZ}) using the so called Hitchin
system (\cite{key-HIT}) for the Dolbeaut space. They conjectured also (and checked
for $G=\mathrm{SL}{\left(2,\mathbb{C}\right)},\,\mathrm{SL}{\left(3,\mathbb{C}\right)}$
using the results by Hitchin and Gothen) that a version of the topological
mirror symmetry holds, i.e., some Hodge numbers $h^{p,q}$ of $\mathcal{M}_{Dol}^{d}{\left(G\right)}$
and $\mathcal{M}_{Dol}^{d}{\left(G^{L}\right)}$, for $G$ and Langlands
dual $G^{L}$, agree. Very recently, Groechenig, Wyss and Ziegler proved the topological mirror symmetry for $G=\mathrm{SL}{\left(n,\mathbb{C}\right)}$ in \cite{key-GWZ}.

Contrarily to $\mathcal{M}_{DR}^{d}{\left(G\right)}$ and $\mathcal{M}_{Dol}^{d}{\left(G\right)}$
cases, the cohomology of $\mathcal{M}_{B}^{d}{\left(G\right)}$ does
not have a pure Hodge structure. This fact motivates the study of
\emph{$E$-polynomials} of the $G$-character varieties.
The $E$-polynomial of a variety $X$ is 
\[
E{\left(X;x,y\right)}:=H_{c}{\left(X;x,y,-1\right)}
\]
where
\[
H_{c}{\left(X;x,y,t\right)}:=\sum{h_{c}^{p,q;j}{\left(X\right)}x^{p}y^{q}t^{j}}
\]
the $h_{c}^{p,q;j}$ being the mixed Hodge numbers with compact support
of $X$ (\cite{key-1}, \cite{key-3}). When the $E$-polynomial only depends on $xy$, we write $E{\left(X;q\right)}$, meaning
\[
E{\left(X;q\right)}=H_{c}{\left(X;\sqrt{q},\sqrt{q},-1\right)}.
\]

Hausel and Rodriguez-Villegas started the computation of the $E$-polynomials
of $G$-character varieties for $G=\mathrm{GL}{\left(n,\mathbb{C}\right)},\,\mathrm{PGL}{\left(n,\mathbb{C}\right)}$
using arithmetic methods inspired on Weil conjectures. In \cite{key-5} they
obtained the $E$-polynomials of $\mathcal{M}_{B}^{d}{\left(\mathrm{GL}{\left(n,\mathbb{C}\right)}\right)}$.
Following this work, in \cite{key-6} Mereb computed the $E$-polynomials of
$\mathcal{M}_{B}^{d}{\left(\mathrm{SL}{\left(n,\mathbb{C}\right)}\right)}$.
He proved also that these polynomials are palindromic and monic.

Another direction of interest is the moduli space of parabolic bundles.
If $p_{1},\ldots,p_{s}$ are $s$ marked points in a Riemann surface
$\Sigma_{g}$ of genus $g$, and $\mathcal{C}_{i}\subseteq G$ semisimple
conjugacy classes for $i=1,\ldots,s$, the corresponding Betti moduli
space of parabolic representations (or \emph{parabolic $G$-character
variety}) is 
\[
\begin{aligned}
\mathcal{M}^{\mathcal{C}_{1},\ldots,\mathcal{C}_{s}}{\left(G\right)}:=\biggl\{&\left(A_{1},B_{1},\ldots,A_{g},B_{g},C_{1},\ldots,C_{s}\right)\in G^{2g+s}\mid \\
& \prod\limits_{\substack{i=1}}^{g}{\left[A_{i}:B_{i}\right]}\prod\limits_{\substack{j=1}}^{s}{C_{j}}=\mathrm{Id}_{G},\,C_{j}\in\mathcal{C}_{j},\,j=1,\ldots,s \biggr\} //G.
\end{aligned}
\]
In \cite{key-Si}, Simpson proved that this space is analytically isomorphic
to the moduli space of flat logarithmic $G$-connections and homeomorphic
to a moduli space of Higgs bundles with parabolic structures at $p_{1},\ldots,p_{s}$.

Hausel, Letellier and Rodriguez-Villegas (\cite{key-HLRV}) found formulae for
the $E$-polynomials of the parabolic character varieties for $G=\mathrm{GL}{\left(n,\mathbb{C}\right)}$
and generic semisimple $\mathcal{C}_{1},\ldots,\mathcal{C}_{s}$.

In this paper, we consider certain parabolic character varieties for
the group $G=\mathrm{Sp}{\left(2n,\mathbb{C}\right)}$. For a semisimple
element $\xi$ belonging to a conjugacy class $\mathcal{C}\subseteq\mathrm{Sp}{\left(2n,\mathbb{C}\right)}$,
we define 
\[
\begin{aligned}
\mathcal{M}_{n}^{\xi}:&=\left\{ \left(A_{1},B_{1},\ldots,A_{g},B_{g}\right)\in\mathrm{Sp}{\left(2n,\mathbb{C}\right)}^{2g}\mid\prod\limits_{\substack{i=1}}^{g}{\left[A_{i}:B_{i}\right]}=\xi\right\} //C_{\mathrm{Sp}{\left(2n,\mathbb{C}\right)}}{\left(\xi\right)}\\
&=\left\{ \left(A_{1},B_{1},\ldots,A_{g},B_{g}\right)\in\mathrm{Sp}{\left(2n,\mathbb{C}\right)}^{2g}\mid\prod\limits_{\substack{i=1}}^{g}{\left[A_{i}:B_{i}\right]}\in\mathcal{C}\right\} //\mathrm{Sp}{\left(2n,\mathbb{C}\right)}
\end{aligned}
\]
where $C_{\mathrm{Sp}{\left(2n,\mathbb{C}\right)}}{\left(\xi\right)}$ is the centralizer of $\xi$ in $\mathrm{Sp}{\left(2n,\mathbb{C}\right)}$. We assume that $\xi$ satisfies the genericity condition \ref{eq:2.1} below;
in particular, $\xi$ is a regular semisimple element, hence $C_{\mathrm{Sp}{\left(2n,\mathbb{C}\right)}}{\left(\xi\right)}=T\cong(\mathbb{C}^{\times})^{n}$,
the maximally split torus in $\mathrm{Sp}{\left(2n,\mathbb{C}\right)}$.
It turns out that $\mathcal{M}_{n}^{\xi}$ is a geometric quotient
and all the stabilisers are finite subgroups of $\boldsymbol{\mu}_{\boldsymbol{2}}^{n}$, the
group of diagonal symplectic involution matrices.

Our goal is to compute the $E$-polynomials of $\mathcal{M}_{n}^{\xi}$
for any genus $g$ and dimension $n$. This is accomplished by arithmetic
methods, following the work of Hausel and Rodriguez-Villegas in \cite{key-5}
and Mereb in \cite{key-6}. Our methods depends on the additive property of
the $E$-polynomial, which allows us to compute this polynomial using
stratifications (see \ref{rem:strata}).

The strategy to compute the $E$-polynomials of $\mathcal{M}_{n}^{\xi}$
is to construct a stratification of $\mathcal{M}_{n}^{\xi}$ and proving
that each stratum $\widetilde{\mathcal{M}}_{n,H}^{\xi}$, with $H$
varying on the set of the subgroups of $\boldsymbol{\mu}_{\boldsymbol{2}}^{n}$ (see Definition \ref{defn:stratum})
has polynomial count, i.e., there is a polynomial $E_{n,H}{\left(q\right)}\in\mathbb{Z}[q]$
such that $\left|\widetilde{\mathcal{M}}_{n,H}^{\xi}{\left(\mathbb{F}_{q}\right)}\right|=E_{n,H}{\left(q\right)}$
for sufficiently many prime powers $q$ in the sense described in
Section \ref{subsection:2.2}. According to Katz's theorem \ref{thm:7}, the $E$-polynomial of
$\widetilde{\mathcal{M}}_{n,H}^{\xi}$ agrees with the counting polynomial
$E_{n,H}$. By Theorem \ref{thm:2.19} below, one reduces to count
\begin{equation}
E_{n}{\left(q\right)}:=\frac{1}{\left(q-1\right)^{n}}\left|\left\{\left(A_{1},B_{1},\ldots,A_{g},B_{g}\right)\in\mathrm{Sp}{\left(2n,\mathbb{F}_{q}\right)}^{2g}\mid\prod\limits_{\substack{i=1}}^{g}{\left[A_{i}:B_{i}\right]}=\xi\right\}\right|\label{eq:counting}
\end{equation}
with $\mathbb{F}_{q}$ a finite field containing the eigenvalues of
$\xi$.
The number of solutions of an equation like \ref{eq:counting} is
given by a Frobenius-type formula involving certain values of the
irreducible characters $\chi$ of $\mathrm{Sp}{\left(2n,\mathbb{F}_{q}\right)}$
(see \ref{prop:frobenius}). Thanks to the formula \ref{eq:1.4} below for the evaluation
of irreducible characters of a finite group of Lie type on a regular semisimple element, the Frobenius
formula and Katz's theorem, we are able to compute $E_{n}{\left(q\right)}$,
hence the $E_{n,H}{\left(q\right)}$'s. Adding them up, we eventually
obtain the following
\begin{thm*}
The $E$-polynomial of $\mathcal{M}_{n}^{\xi}$ satisfies
\[
E{\left(\mathcal{M}_{n}^{\xi};q\right)}=E_{n}{\left(q\right)}=\frac{1}{\left(q-1\right)^{n}}\sum\limits_{\substack{\tau}}{\left(H_{\tau}{\left(q\right)}\right)^{2g-1}C_{\tau}}.
\]
\end{thm*}
Here, $H_{\tau}{\left(q\right)}$ are polynomials with integer coefficients
(see \ref{eq:1.12}), $C_{\tau}$ are integer constants and the sum is over a well described set (see \ref{defn:type}).
It is remarkable that the $E$-polynomial of $\mathcal{M}_{n}^{\xi}$
does not depend on the choice of the generic element $\xi$. 
A direct consequence of our calculation and of the fact that $\mathcal{M}_{n}^{\xi}$ is equidimensional (see Corollary \ref{cor:2.14}) is the following
\begin{cor*}
The $E$-polynomial of $\mathcal{M}_{n}^{\xi}$ is palindromic
and monic. In particular, the parabolic character variety $\mathcal{M}_{n}^{\xi}$
is connected.
\end{cor*}
\noindent Our formula also implies
\begin{cor*}
The Euler characteristic $\chi{\left(\mathcal{M}_{n}^{\xi}\right)}$
of $\mathcal{M}_{n}^{\xi}$ vanishes for $g>1$. For
$g=1$, we have 
\[
\sum\limits_{\substack{n\geq 0}}{\frac{\chi{\left(\mathcal{M}_{n}^{\xi}\right)}}{2^{n}n!}T^{n}}=\prod\limits_{\substack{k\geq 1}}{\frac{1}{\left(1-T^{k}\right)^{3}}}.
\]
\end{cor*}
\noindent
See Corollaries \ref{cor:Euler}, \ref{cor:pal}, \ref{cor:con} and \ref{cor:char} for details.

\noindent For $n=1$, the formula looks like:
\[
\begin{aligned}
E{\left(\mathcal{M}_{1}^{\xi};q\right)}=&\left(q^{3}-q\right)^{2g-2}\left(q^{2}+q\right)+\left(q^{2}-1\right)^{2g-2}\left(q+1\right)\\
&+\left(2^{2g}-2\right)\left(q^{2}-q\right)^{2g-2}q.
\end{aligned}
\]
This result recovers the ones obtained by Logares, Muñoz
and Newstead in \cite{key-LOG} for small genus $g$ and by Martinez and Muñoz
in \cite{key-15} for all possible $g$.

The present article is organized as follows: In Section \ref{section:2}, we go
over the basics of combinatorics and representation theory that are
going to be needed. In Section \ref{section:3}, we study the geometry of the
parabolic character varieties to be studied. In Section \ref{section:4}, we perform
the computation of the $E$-polynomial of $\mathcal{M}_{n}^{\xi}$
and prove the corollaries concerning the topological properties of
$\mathcal{M}_{n}^{\xi}$ encoded in $E{\left(\mathcal{M}_{n}^{\xi};q\right)}$.
\newline

\noindent\textbf{Acknowledgements}. The author thanks Fernando Rodriguez-Villegas for his invaluable help and for supervising this work, Emmanuel Letellier, Martin Mereb and Meinolf Geck for useful conversations and discussions, C\'edric Bonnaf\'e for helpful email correspondence. 

\section{Preliminaries}
\label{section:2}
\subsection{Mixed Hodge structures}

Motivated by the (then still unproven) Weil Conjectures and Grothendieck's
``yoga of weights'', which drew cohomological conclusions about
complex varieties from the truth of those conjectures, Deligne in \cite{key-1} and \cite{key-3} proved the existence of \emph{mixed Hodge structures} on the cohomology of a complex algebraic variety.
\begin{prop}
\label{prop:Proposition 1}\emph{(\cite{key-1}, \cite{key-3})}. Let $X$ be a complex algebraic variety.
For each $j$ there is an increasing weight filtration
\begin{equation}
0=W_{-1}\subseteq W_{0}\subseteq\cdots\subseteq W_{2j}=H^{j}{\left(X,\mathbb{Q}\right)}
\end{equation}
and a decreasing Hodge filtration
\begin{equation}
H^{j}{\left(X,\mathbb{C}\right)}=F^{0}\supseteq F^{1}\supseteq\cdots\supseteq F^{m}\supseteq F^{m+1}=0
\end{equation}
such that the filtration induced by $F$ on the complexification of
the graded pieces $Gr_{l}^{W}:=W_{l}/W_{l-1}$ of the weight filtration
endows every graded piece with a pure Hodge structure of weight $l$,
or equivalently, for every $0\leq p\leq l$, we have 
\begin{equation}
Gr_{l}^{W^{\mathbb{C}}}=F^{p}Gr_{l}^{W^{\mathbb{C}}}\oplus\overline{F^{l-p+1}Gr_{l}^{W^{\mathbb{C}}}}.
\end{equation}
\end{prop}
This mixed Hodge structure of $X$ respects most operations
in cohomology, like maps $f^{\ast}:H^{\ast}{\left(Y,\mathbb{Q}\right)}\rightarrow H^{\ast}{\left(X,\mathbb{Q}\right)}$
induced by a morphism of varieties $f:X\rightarrow Y$, maps induced
by field automorphisms $\sigma\in\mathrm{Aut}{\left(\mathbb{C}/\mathbb{Q}\right)}$,
the Künneth isomorphism
\begin{equation}
H^{\ast}{\left(X\times Y,\mathbb{Q}\right)}\cong H^{\ast}{\left(X,\mathbb{Q}\right)}\otimes H^{\ast}{\left(Y,\mathbb{Q}\right)},
\end{equation}
cup products, etc.

Using Deligne's construction \cite[8.3.8]{key-3} of mixed Hodge structure on relative
cohomology, one can define (\cite{key-2}) a well-behaved mixed Hodge structure
on compactly supported cohomology $H_{c}^{\ast}{\left(X,\mathbb{Q}\right)}$,
compatible with Poincaré duality for smooth connected $X$ (see also \cite{key-4}).
\begin{defn}
Define the \emph{compactly supported mixed Hodge numbers} by
\begin{equation}
h_{c}^{p,q;j}{\left(X\right)}:=\mathrm{dim}_{\mathbb{C}}{\left(Gr_{p}^{F}Gr_{p+q}^{W^{\mathbb{C}}}H_{c}^{j}{\left(X,\mathbb{C}\right)}\right)}.
\end{equation}
Form the \emph{compactly supported mixed Hodge polynomial}:
\begin{equation}
H_{c}{\left(X;x,y,t\right)}:=\sum\limits_{\substack{p,q,j}}{h_{c}^{p,q;j}{\left(X\right)}x^{p}y^{q}t^{j}}
\end{equation}
and the $E$-\emph{polynomial} of $X$:
\begin{equation}
E{\left(X;x,y\right)}:=H_{c}{\left(X;x,y,-1\right)}.
\end{equation}
\end{defn}
\begin{rem}
\label{rem:3}By definition, we can deduce the following properties
of the $E$-polynomial $E{\left(X;x,y\right)}$ of an algebraic variety
$X$:
\begin{itemize}
\item $E{\left(X;1,1\right)}=\chi{\left(X\right)}$, the Euler characteristic
of $X$.
\item The total degree of $E{\left(X;x,y\right)}$ is twice the dimension
of $X$ as a complex algebraic variety.
\item If $X$ is a smooth algebraic variety, the coefficient of $x^{\mathrm{dim}{\left(X\right)}}y^{\mathrm{dim}{\left(X\right)}}$ in
$E{\left(X;x,y\right)}$ is the number of the highest dimensional connected
components of $X$. 
\end{itemize}
\end{rem}
\begin{rem}
\label{rem:strata}
If $\left\{Z_{i}\right\} _{i=1,\ldots,n}$ is a stratification of
an algebraic variety $X$, i.e., a finite partition of $X$ into the locally closed subsets $Z_{i}$, then 
\begin{equation}
E{\left(X;x,y\right)}=\sum\limits_{\substack{i=1}}^{n}{E{\left(Z_{i};x,y\right)}}
\end{equation}
i.e., the $E$-polynomial is additive under stratifications. 
\end{rem}

\subsection{Spreading out and Katz's theorem}
\label{subsection:2.2}

Sometimes, the $E$-polynomial could be calculated using arithmetic
algebraic geometry. The setup is the following.
\begin{defn}
Let $X$ be a complex algebraic variety, $R$ a finitely generated
$\mathbb{Z}$-algebra, $\phi:\,R\hookrightarrow\mathbb{C}$ a fixed
embedding. We say that a separated scheme $\mathfrak{X}/R$ is a \emph{spreading
out} of $X$ if its extension of scalars $\mathfrak{X}_{\phi}$ is
isomorphic to $X$. 
\end{defn}

\begin{defn}
Suppose that a complex algebraic variety $X$ has a spreading out
$\mathfrak{X}$ such that for every ring homomorphism $\psi:\,R\rightarrow\mathbb{F}_{q}$,
the number of points of $\mathfrak{X}_{\psi}{\left(\mathbb{F}_{q}\right)}$
is given by $P_{X}{\left(q\right)}$ for some fixed $P_{X}{\left(t\right)}\in\mathbb{Z}[t]$.
We say that $X$ is a \emph{polynomial count variety} and that $P_{X}$
is the counting polynomial.
\end{defn}
Then we have the following fundamental result:
\begin{thm}
\label{thm:7}
\emph{(\cite[Katz (2.18)]{key-5})}. Let $X$ be a variety
over $\mathbb{C}$. Assume $X$ is polynomial count with counting
polynomial $P_{X}{\left(t\right)}\in\mathbb{Z}[t]$, then
the $E$-polynomial of $X$ is given by $E{\left(X;x,y\right)}=P_{X}{\left(xy\right)}$.
\end{thm}
In this case, and more generally, when the $E$-polynomial only depends on $xy$, we write
\[
E{\left(X;q\right)}:=E{\left(X;\sqrt{q},\sqrt{q}\right)}.
\]
\begin{rem}
\label{rem:propr}
By \ref{rem:3}, for a variety $X$ whose $E$-polynomial is given by $P_{X}{\left(q\right)}$, the Euler characteristic of $X$ is equal
to $P_{X}{\left(1\right)}$, while if $X$ is smooth, the leading coefficient of $P_{X}{\left(q\right)}$
is the number of highest dimensional connected components of $X$.
\end{rem}

\begin{rem}
Informally, Katz's theorem says that if we can count the number of
solutions of the equations defining our variety over $\mathbb{F}_{q}$,
and this number turns out to be some universal polynomial evaluated
in $q$, then this polynomial determines the $E$-polynomial of the
variety.
\end{rem}
Actually, it is enough for this to be true not necessarily for all
finite fields. In fact, one can restrict the computations via a suitable
choice of the finitely generated $\mathbb{Z}$-algebra which a spreading
out of the variety $X$ is defined over.
\begin{example}
(\cite[Example 2.4]{key-6}). Let us take the affine curve 
\[
C=\left\{ \left(x,y\right)\in\mathbb{C}^{2}\mid2x^{2}+3y^{2}=5\right\} .
\]
We want a scheme $\mathcal{X}$ defined by the same equation over
a finitely generated $\mathbb{Z}$-algebra. It is easy to check that
$\left|C{\left(\mathbb{F}_{p}\right)}\right|=p-\left(\frac{-6}{p}\right)$
for $p\gneq5$ prime, where $\left(\frac{-6}{p}\right)$ is the Legendre
symbol, and that $\left|C{\left(\mathbb{F}_{2}\right)}\right|=2$, $\left|C{\left(\mathbb{F}_{3}\right)}\right|=6$
and $\left|C{\left(\mathbb{F}_{5}\right)}\right|=9$. To have a polynomial
count for $C$, we need to exclude some primes. To get rid of 2, 3
and 5, we consider the scheme $\mathcal{X}$ over $\mathbb{Z}{\left[\frac{1}{30}\right]}$,
ending up with a quasi-polynomial, since the term $\left(\frac{-6}{p}\right)$
is periodic. To satisfy the hypotheses of Theorem \ref{thm:7}, we
still have to exclude all primes $p$ such that $-6$ is a quadratic
non-residue modulo $p$. This can be accomplished by adding $\sqrt{-6}$
to the base ring. The scheme $\mathcal{X}/\mathbb{Z}{\left[\frac{1}{30},\sqrt{-6}\right]}$
is a spreading out for $C$ with polynomial count $P_{C}{\left(p\right)}=p-1$.
By Katz's result, the $E$-polynomial of $C{\left(\mathbb{C}\right)}$
is $E\left(C;x,y\right)=xy-1$. This is consistent with the fact that $C\cong\mathbb{C}^{\times}$.
\end{example}

\subsection{The poset of partitions}

We collect in this section some notations, concepts and results on
partitions of sets that we will need later. The main references are
\cite{key-24} and \cite{key-25}.

Let $[c]:=\left\{1,\ldots,c\right\} $, $c\in\mathbb{N}$
and let $\Pi_{c}$ be the poset of partitions of $[c]$;
it consists of all decomposition $\pi$ of $[c]$ into disjoint
unions of non-empty subsets $[c]=\coprod\limits_{j=1}^l{I_{j}}$
ordered by refinement, which we denote by $\preceq$. Concretely,
$\pi\preceq\pi^{\prime}$ in $\Pi_{c}$ if every subset in $\pi$
is a subset of one in $\pi^{\prime}$. We call the $I_{j}$'s the
\emph{blocks} of $\pi$ and the integer $l$ the \emph{length} of $\pi$,
which we denote by $l{\left(\pi\right)}$.

If $x\in\mathbb{N}$ and $h:[m]\rightarrow[x]$,
define the \emph{kernel} of $h$, and denote it by $Ker{\left(h\right)}$,
to be the partition of $[c]$ induced by the equivalence
relation
\[
a\equiv b\Longleftrightarrow h{\left(a\right)}=h{\left(b\right)}.
\]
Now, fix $\pi\in\Pi_{c}$ and let
\begin{equation}
\begin{aligned}
\label{eq:sigma}
&\Sigma^{\prime}_{\pi}:=\left\{ h:[c]\rightarrow[x]\mid Ker{\left(h\right)}=\pi\right\},\\
&\Sigma_{\pi}:=\left\{ h:[c]\rightarrow[x]\mid\text{$Ker{\left(h\right)}=\sigma$ for some $\sigma\succeq\pi$}\right\}.
\end{aligned}
\end{equation}
If we define $f{\left(\pi\right)}:=\left|\Sigma_{\pi}\right|$ and $g{\left(\pi\right)}:=\left|\Sigma_{\pi}^{\prime}\right|$, then clearly $g{\left(\pi\right)}=\left(x\right)_{\left(l{\left(\pi\right)}\right)}:=x{\left(x-1\right)}\cdots{\left(x-l{\left(\pi\right)}+1\right)}$
and $f\left(\pi\right)=x^{l\left(\pi\right)}$. On the other hand,
since
\[
f{\left(\pi\right)}=\sum\limits_{\substack{\pi\preceq\sigma}}{g{\left(\sigma\right)}},
\]
by Möbius inversion we have
\[
g{\left(\pi\right)}=\sum\limits_{\substack{\pi\preceq\sigma}}{\mu{\left(\pi,\sigma\right)}f{\left(\sigma\right)}}
\]
where $\mu$ is the \emph{Möbius function} of the poset $\Pi_{c}$;
in other words,
\begin{equation}
\left(x\right)_{l{\left(\pi\right)}}=\sum\limits_{\substack{\pi\preceq\sigma}}{\mu{\left(\pi,\sigma\right)}x^{l{\left(\sigma\right)}}}.\label{eq:1}
\end{equation}

Since \ref{eq:1} holds for any $x\in\mathbb{N}$ and since both sides
of \ref{eq:1} are polynomials, it is a polynomial identity. In particular,
if we specialize it at $x=-1$, we obtain 
\begin{equation}
\left(-1\right)^{l\left(\pi\right)}\left(l\left(\pi\right)\right)!=\sum\limits_{\substack{\pi\preceq\sigma}}{\left(-1\right)^{l{\left(\sigma\right)}}\mu{\left(\pi,\sigma\right)}}.\label{eq:2}
\end{equation}

\subsection{Representation theory}

In this section, we list the facts from representation theory we need.
Before doing it, let us fix some notations.

If $H$ is a finite group, we denote by $\mathrm{Irr}{\left(H\right)}$
the set of the irreducible characters of the finite dimensional representations
of $H$; the natural inner product $\left\langle \,,\,\right\rangle $
for characters is given by 
\[
\left\langle \chi,\chi^{\prime}\right\rangle _{H}:=\frac{1}{\left|H\right|}\sum\limits_{\substack{h\in H}}{\chi{\left(h\right)}\chi^{\prime}{\left(h^{-1}\right)}}
\]
and we define the \emph{dual group} of $H$ as $\widehat{H}:=\mathrm{Hom}{\left(H,\mathbb{C}^{\times}\right)}$.
When $H$ is abelian, $\widehat{H}=\mathrm{Irr}{\left(H\right)}$. 

\subsubsection{Deligne-Lusztig characters}

In these subsections, we refer to \cite{key-7}, \cite{key-23} and \cite{key-22} for basic definitions and proofs.

Assume that $G$ is a connected reductive linear algebraic group defined
over $\overline{\mathbb{F}_{q}}$, the algebraic closure of a finite
field $\mathbb{F}_{q}$ of odd characteristic $p$, $F:G\rightarrow G$ a Frobenius map, $G^{F}$ the finite group of Lie type
of the rational points of $G$. In particular, $G^{F}$ is a finite
group with a split $BN$-pair at characteristic $p$. If
$T$ is an $F$-stable maximal torus of $G$, $\theta\in\widehat{T^{F}}$,
we denote by $R_{T}^{G}{\left(\theta\right)}$ the corresponding \emph{Deligne-Lusztig
character} (see \cite{key-8}). Recall the following
\begin{prop}
\emph{\label{prop:11}(\cite[Theorem 6.8]{key-8}). }Let $T$ be a $F$-stable
maximal torus of $G$, $\theta,\theta^{\prime}\in\widehat{T^{F}}$. Then
\begin{equation}
\left\langle R_{T}^{G}{\left(\theta\right)},R_{T}^{G}{\left(\theta^{\prime}\right)}\right\rangle _{G^{F}}=\frac{\left|\left\{n\in N_{G}{\left(T\right)}^{F}\mid\theta^{n}=\theta^{\prime}\right\}\right|}{\left|T^{F}\right|}\label{eq:3}
\end{equation}
where $N_{G}{\left(T\right)}$ is the normalizer of $T$ in $G$, $\theta^{n}{\left(t\right)}:=\theta{\left(n^{-1}tn\right)}$
for every $t\in T^{F}$. Moreover, either $\left\langle R_{T}^{G}{\left(\theta\right)},R_{T}^{G}{\left(\theta^{\prime}\right)}\right\rangle _{G^{F}}=0$
or $R_{T}^{G}{\left(\theta\right)}=R_{T}^{G}{\left(\theta^{\prime}\right)}$. 
\end{prop}

Let $s\in G^{F}$ be a regular semisimple element, i.e., an element
of $G^{F}$ representable by a semisimple matrix with all distinct
eigenvalues via a closed embedding of $G$ into $\mathrm{GL}{\left(n,\overline{\mathbb{F}_{q}}\right)}$
for some $n\in\mathbb{N}$, $T$ the unique maximal torus containing $s$. We have 
\begin{prop}
\emph{(\cite[(7.6.2)]{key-8}) }If $\chi\in\mathrm{Irr}{\left(G^{F}\right)}$, then
\begin{equation}
\chi{\left(s\right)}=\sum\limits_{\substack{\theta\in\widehat{T^{F}}}}{\theta{\left(s^{-1}\right)}{\left\langle R_{T}^{G}{\left(\theta\right)},\chi\right\rangle _{G^{F}}}}.\label{eq:1.4}
\end{equation}
\end{prop}

\subsubsection{Principal series representations}

Let us consider a maximally split torus $T$ in $G$, i.e., an $F$-stable
torus contained in a $F$-stable Borel subgroup $B$. Then, $B^{F}=T^{F}\ltimes U^{F}$,
with $U$ the unipotent radical of $B$. If $\theta\in\widehat{T^{F}}$,
we have that $R_{T}^{G}{\left(\theta\right)}=\mathrm{Ind}_{B^{F}}^{G^{F}}{\left(\widetilde{\theta}\right)}$,
where $\widetilde{\theta}:=\theta\circ p_{1}\in\widehat{B^{F}}$ and
$p_{1}:B^{F}=T^{F}\ltimes U^{F}\twoheadrightarrow T^{F}$ is the natural
projection onto $T^{F}$. The irreducible components of $R_{T}^{G}{\left(\theta\right)}$
are called the \emph{principal series associated to $\left(T,\theta\right)$}.\\
\\
$\mathbf{Notation}$: In the following, sometimes we may use the symbol
$R_{T}^{G}{\left(\theta\right)}$ also to denote both the associated representation and the set of principal series
of $\left(T,\theta\right)$. So if $\chi$ is an irreducible constituent
of $R_{T}^{G}\left(\theta\right)$, we write $\chi\in R_{T}^{G}{\left(\theta\right)}$. 
\begin{rem}
\label{rem:13}
If $\theta,\theta^{\prime}\in\widehat{T^{F}}$, since
$R_{T}^{G}{\left(\theta\right)}$ and $R_{T}^{G}{\left(\theta^{\prime}\right)}$
are characters of representations of $G^{F}$, we deduce from \ref{prop:11}
that either the principal series of $\left(T,\theta\right)$
and $\left(T,\theta^{\prime}\right)$ coincide or they are
disjoint. So we have that $R_{T}^{G}{\left(\theta\right)}\cap R_{T}^{G}{\left(\theta^{\prime}\right)}\neq\emptyset$
if and only if there exists $w\in W^{F}$ such that $\theta^{w}=\theta^{\prime}$.
Here, $W^{F}$ is the group of rational points of the Weyl group $W:=N_{G}{\left(T\right)}/T$
of $G$, acting by conjugation on $\widehat{T^{F}}$.
\end{rem}
If $S_{\theta}:=\mathrm{Stab}_{W^{F}}{\left(\theta\right)}$, there is an isomorphism
between $\mathrm{End}_{G^{F}}{\left(R_{T}^{G}\left(\theta\right)\right)}$
and the group algebra $\mathbb{C}S_{\theta}$. In order to establish it,
we need to recall some definitions and results that can be found in
\cite{key-9}.

Let $\Phi$ be the set of roots of $G$; for every $\alpha\in\Phi$,
let $q_{\alpha}{\left(\theta\right)}:=q^{c_{\alpha}{\left(\theta\right)}}$,
$c_{\alpha}{\left(\theta\right)}\in\mathbb{Z}_{\geq0}$, be the parameters
defined in \cite[Lemma 2.6]{key-9}. Define the set 
\begin{equation}
\Gamma:=\left\{ \beta\in\Phi\mid q_{\alpha}{\left(\theta\right)}\neq1\right\} \label{eq:1.5-1}
\end{equation}
and let $W_{S_{\theta}}$ be the group generated by the reflections
corresponding to roots in $\Gamma$. Thus, $\Gamma$ is the root system
of the reflection group $W_{S_{\theta}}$. If 
\begin{equation}
D:=\left\{ w\in S_{\theta}\mid w{\left(\alpha\right)}\in\Phi^{+},\,\forall\alpha\in\Gamma^{+}\right\} \label{eq:1.6-1}
\end{equation}
then, \cite[Lemma 2.9]{key-9}, $D$ is an abelian $p^{\prime}$-group
normalizing $W_{S_{\theta}}$ and $S_{\theta}=D\ltimes W_{S_{\theta}}$.
Moreover, let $\Sigma$ be the set of simple roots of $\Gamma$ consisting
of roots that are positive in $\Phi$.
\begin{defn}
\label{def:14}
The \emph{generic algebra} $\mathcal{A}{\left(u\right)}$
is the algebra over $\mathbb{C}[u]$ with basis $\left\{a_{w}\mid w\in S_{\theta}\right\} $
such that, if $w\in S,$ $d\in D$ and $s$ is the reflection corresponding
to the root $\beta\in\Sigma$, the following relations hold:
\begin{enumerate}
\item \label{enu:1}
$a_{d}a_{w}=a_{dw},\,a_{w}a_{d}=a_{wd}$.
\item \label{enu:2}
$a_{w}a_{s}=a_{ws}$ if $w{\left(\beta\right)}\in\Gamma^{+}$.
\item \label{enu:3}
$a_{w}a_{s}=u_{\beta}{\left(\theta\right)}a_{ws}+\left(u_{\beta}{\left(\theta\right)}-1\right)a_{w}$
if $w{\left(\beta\right)}\in\Gamma^{-}$, $u_{\beta}{\left(\theta\right)}:=u^{c_{\beta}{\left(\theta\right)}}$.
\end{enumerate}
\end{defn}

For any ring $K$ such that $\mathbb{C}[u]\subseteq K$,
write $\mathcal{A}{\left(u\right)}^{K}=\mathcal{A}{\left(u\right)}\otimes_{\mathbb{C}[u]}K$.
An algebra homomorphism $f:\mathbb{C}[u]\rightarrow\mathbb{C}$
makes $\mathbb{C}$ into a $\left(\mathbb{C},\mathbb{C}[u]\right)$-bimodule
via $\left(a,p\right)\cdot c:=acf{\left(p\right)}$, $a,c\in\mathbb{C}$, $p\in\mathbb{C}[u]$.
If $f{\left(u\right)}=b,$ the \emph{specialization $\mathcal{A}{\left(b\right)}:=\mathcal{A}{\left(u\right)}\otimes_{f}\mathbb{C}$} is an algebra over $\mathbb{C}$ with basis $\left\{ a_{w}\otimes1\mid w\in S\right\} $ whose members satisfy relations \ref{enu:1}, \ref{enu:2} and \ref{enu:3} in Definition \ref{def:14},
after replacing $u$ with $b$. 
\begin{thm}
\label{thm:15}
\emph{(Tits, \cite[Theorem 1.11]{key-10}).} Let $K=\mathbb{C}{\left(u\right)}$
be the quotient field of $\mathbb{C}[u]$. Then $\mathcal{A}{\left(u\right)}^{K}$ is a separable $K$-algebra and for each $b\in\mathbb{C}$ such that
$\mathcal{A}{\left(b\right)}$ is separable (and so semisimple), the
algebras $\mathcal{A}{\left(u\right)}^{K}$ and $\mathcal{A}{\left(b\right)}$
have the same numerical invariants.
\end{thm}
\begin{cor}
\label{cor:16} 
$\mathrm{End}_{G^{F}}{\left(R_{T}^{G}{\left(\theta\right)}\right)}\cong\mathbb{C}S_{\theta}$
as $\mathbb{C}$-algebras.
\end{cor}
\begin{proof}
The algebras $\mathcal{A}{\left(q\right)}$ and $\mathcal{A}{\left(1\right)}$
are respectively isomorphic to $\mathrm{End}_{G^{F}}{\left(R_{T}^{G}{\left(\theta\right)}\right)}$
and $\mathbb{C}S_{\theta}$ by \cite[Theorem 2.17, 2.18]{key-10} and
since they are both semisimple, they have the same numerical invariants
by \ref{thm:15} and so are isomorphic.
\end{proof}
\begin{cor}
\label{cor:17}
The irreducible components of $R_{T}^{G}{\left(\theta\right)}$
are in bijective correspondence with the irreducible representations
of the algebra $\mathbb{C}S_{\theta}$ .
\end{cor}
The correspondence established in \ref{cor:17} can be stated more
precisely by the following 
\begin{prop}
\label{prop:18}
\emph{(Lemma 3.4 in \cite{key-9}).} Let $\chi$ be an irreducible
character of $\mathcal{A}{\left(u\right)}^{K}$. Then for all
$w\in S_{\theta}$, $\chi{\left(a_{w}\right)}$ is in the integral closure
of $\mathbb{C}[u]$ in $\overline{\mathbb{C}{\left(u\right)}}$.
Let $f:\mathbb{C}[u]\rightarrow\mathbb{C}$ be a homomorphism
such that $f(u)=b$ and $\mathcal{A}{\left(b\right)}$ is
separable, and let $f^{\ast}$ be an extension of $f$ to the integral
closure of $\mathbb{C}[u]$. Then the linear map $\chi_{f}:\mathcal{A}{\left(b\right)}\rightarrow\mathbb{C}$
defined by $\chi_{f}{\left(a_{w}\otimes1\right)}:=f^{\ast}{\left(\chi{\left(a_{w}\right)}\right)}$,
for all $w\in S_{\theta}$, is an irreducible character of $\mathcal{A}{\left(b\right)}$.
For a fixed extension $f^{\ast}$ of $f$, the map $\chi\longmapsto\chi_{f}$
is a bijection between the irreducible characters of $\mathcal{A}{\left(u\right)}^{K}$ and those of $\mathcal{A}{\left(b\right)}$.
\end{prop}

From \cite[Theorem A]{key-20}, we deduce this important result on the multiplicities
of the principal series representations.
\begin{prop}
\label{prop:19}If $\chi\in R_{T}^{G}{\left(\theta\right)}$
corresponding to $\beta\in\mathrm{Irr}{\left(S_{\theta}\right)}$, then
\begin{equation}
\left\langle \chi,R_{T}^{G}{\left(\theta\right)}\right\rangle _{G^{F}}=\beta{\left(1\right)}.\label{eq:1.5}
\end{equation}
\end{prop}
Let $\overline{\beta}\in\mathrm{Irr}{\left(\mathcal{A}{\left(u\right)}^{K}\right)}$
with $K=\mathbb{C}{\left(u\right)}$, $\beta$ the corresponding character
of $S_{\theta}$. Define 
\begin{equation}
D_{\beta}{\left(u\right)}:=\frac{\overline{\beta}{\left(1\right)}P{\left(u\right)}}{\sum\limits_{\substack{w\in S_{\theta}}}{u_{w}{\left(\theta\right)}^{-1}\overline{\beta}{\left(a_{w^{-1}}\right)}\overline{\beta}{\left(a_{w}\right)}}}\label{eq:1.6}
\end{equation}
where $u_{w}{\left(\theta\right)}:=\prod\limits_{\substack{\alpha\in\Gamma^{+}\\w\left(\alpha\right)\in\Gamma^{-}}}u_{\alpha}{\left(\theta\right)}$
and $P{\left(u\right)}:=\sum\limits_{\substack{w\in W^{F}}}{u_{w}{\left(1\right)}}$,
the \emph{Poincaré polynomial} of the Coxeter group $W^{F}$. 
\begin{rem}
\label{rem:20}
If $\chi$ is the corresponding irreducible component
of $R_{T}^{G}{\left(\theta\right)}$, then $D_{\beta}{\left(q\right)}=\chi{\left(1\right)}$
(see \cite{key-9}).
\end{rem}

Let $\mathcal{B}{\left(u\right)}$ be the subalgebra of $\mathcal{A}\left(u\right)$
generated by $\left\{a_{w}\mid w\in W_{S_{\theta}}\right\} $. Then
$\mathcal{B}{\left(u\right)}$ is the generic algebra corresponding
to the Coxeter group $W_{S_{\theta}}$ and by Proposition \ref{prop:18},
there is a bijection between $\mathrm{Irr}{\left(\mathcal{B}{\left(u\right)}^{K}\right)}$
and $\mathrm{Irr}{\left(W_{S_{\theta}}\right)}$. For an irreducible
character $\overline{\varphi}$ of $\mathcal{B}{\left(u\right)}^{K}$, and the
corresponding character $\varphi$ of $W_{S_{\theta}}$,
the \emph{generic degree} is defined as follows:
\begin{equation}
d_{\varphi}{\left(u\right)}:=\frac{\overline{\varphi}{\left(1\right)}P_{\theta}{\left(u\right)}}{\sum\limits_{\substack{w\in W_{S_{\theta}}}}{u_{w}{\left(\theta\right)}^{-1}\overline{\varphi}{\left(a_{w^{-1}}\right)}\overline{\varphi}{\left(a_{w}\right)}}}\label{eq:1.7-2}
\end{equation}
where $P_{\theta}{\left(u\right)}:=\sum\limits_{\substack{w\in W_{S_{\theta}}}}{u_{w}{\left(\theta\right)}}$
is the Poincaré polynomial of the Coxeter group $W_{S_{\theta}}$. 
\begin{rem}
\label{rem:21-1}It is true that $d_{\varphi}{\left(1\right)}=\varphi{\left(1\right)}$ (\cite[Theorem 5.7]{key-10}),
$P_{\theta}{\left(u\right)}$ divides $P{\left(u\right)}$ (\cite[Proposition 35]{key-16}) and that $d_{\varphi}{\left(u\right)}=\frac{1}{c_{\varphi}}f_{\varphi}$
where $f_{\varphi}\in\mathbb{Z}[u]$ is a monic polynomial and $c_{\varphi}\in\mathbb{N}$, both depending on $\varphi$ (\cite[Corollary 9.3.6]{key-11}). 
\end{rem}
The group $D$, defined in \ref{eq:1.6-1}, acts as a group of automorphisms
of $\mathcal{B}{\left(u\right)}^{K}$ via $a_{w}\mapsto a_{dwd^{-1}}$,
$d\in D$, $w\in W_{S_{\theta}}$. Thus, for each $d\in D$, if $\overline{\varphi}\in\mathrm{Irr}{\left(\mathcal{B}{\left(u\right)}^{K}\right)}$,
the character $\overline{\varphi}^{d}$ of $\mathcal{B}{\left(u\right)}^{K}$determined
by $\overline{\varphi}^{d}{\left(a_{w}\right)}:=\overline{\varphi}{\left(a_{dwd^{-1}}\right)}$
is irreducible too. 
\begin{prop}
\label{prop:20-1}\emph{(\cite[Theorem 3.13]{key-9}). }Let $\overline{\beta}\in\mathrm{Irr}{\left(\mathcal{A}{\left(u\right)}^{K}\right)}$, $\beta$ the corresponding character in $S_{\theta}$,
$\overline{\varphi}$ an irreducible component of $\mathrm{Res}_{\mathcal{B}{\left(u\right)}^{K}}^{\mathcal{A}{\left(u\right)}^{K}}{\left(\overline{\beta}\right)}$, $\varphi$ the corresponding character in $W_{S_{\theta}}$
and $C:=\left\{d\in D\mid\overline{\varphi}^{d}=\overline{\varphi}\right\}$. Then 
\begin{equation}
D_{\beta}{\left(u\right)}=\frac{P{\left(u\right)}}{P_{\theta}{\left(u\right)}\left|C\right|}d_{\varphi}{\left(u\right)}.\label{eq:1.6-2}
\end{equation}
\end{prop}
\begin{rem}
\label{rem:21}
One can show that, if $\overline{\beta}$ and $\overline{\varphi}$ are
as in Proposition \ref{prop:20-1}, then $\mathrm{Res}_{\mathcal{B}{\left(u\right)}^{K}}^{\mathcal{A}{\left(u\right)}^{K}}{\left(\overline{\beta}\right)}=\frac{1}{\left|C\right|}\sum\limits_{\substack{d\in D}}{\overline{\varphi}^{d}}$
(\cite[proof of Theorem 3.13]{key-9}). Thus, by Proposition \ref{prop:18}, we have that 
\begin{equation}
\label{eq:index1}
\beta{\left(1\right)}=\frac{\left|D\right|}{\left|C\right|}\varphi{\left(1\right)}.
\end{equation}
Then Proposition \ref{prop:20-1} tells us that 
\begin{equation}
\label{eq:index2}
D_{\beta}{\left(1\right)}=\left[W^{F}:S_{\theta}\right]{\beta{\left(1\right)}}.
\end{equation}
\end{rem}

\subsubsection{Principal series representations of $\mathrm{Sp}{\left(2n,\mathbb{F}_{q}\right)}$}\label{subsec:1}

Now we deal with the case 
\[
G=\mathrm{Sp}{\left(2n,\overline{\mathbb{F}_{q}}\right)}:=\left\{ A\in\mathfrak{gl}{\left(2n,\overline{\mathbb{F}_{q}}\right)}\mid A^{t}JA=A\right\} 
\]
where 
\[
J=\begin{pmatrix}
 & & & & & 1\\
 & & & & \iddots &\\
 & & & 1 & &\\
 & & -1 & & &\\
 & \iddots & & & &\\
-1 & & & & &
\end{pmatrix}
\]
relatively to a basis $\overrightarrow{\mathcal{B}}=\left\{e_{1},\ldots,e_{n},e_{-n},\ldots,e_{-1}\right\}$, so in this subsection, $G^{F}=\mathrm{Sp}(2n,\mathbb{F}_{q})$.
Let us consider the maximally split torus 
\[
T=\left\{ \mathrm{diag}{\left(\lambda_{1},\ldots,\lambda_{n},\lambda_{n}^{-1},\ldots,\lambda_{1}^{-1}\right)}\mid\lambda_{i}\in\overline{\mathbb{F}_{q}}^{\times}\right\}.
\]
Therefore, $T^{F}\cong\left(\mathbb{F}_{q}^{\times}\right)^{n}$ and
$\widehat{T^{F}}\cong\mathbb{Z}_{q-1}^{n}$ because every $\theta\in\widehat{T^{F}}$
is characterized (not canonically) by a $n$-tuple of exponents modulo $q-1$. The Weyl
group $W_{n}$ is isomorphic to $S_{n}\ltimes\boldsymbol{\mu}_{\boldsymbol{2}}^{n}$, where
$\boldsymbol{{\mu}_{2}}:=\left\{ 1,-1\right\} $ and $S_{n}$ is the symmetric group, so $W_{n}$ is a Coxeter group of type $B_{n}$. It turns out that $W_{n}^{F}=W_{n}$, and that $W_{n}$ acts on $\widehat{T^{F}}$ 
as follows:
\begin{equation}
\begin{aligned}
\label{eq:1.7-1}
\varphi:\left(S_{n}\ltimes\boldsymbol{\mu}_{\boldsymbol{2}}^{n}\right)\times\widehat{T^{F}}&\rightarrow\widehat{T^{F}}\\
\left(\left(\sigma,\left(\varepsilon_{1},\ldots,\varepsilon_{n}\right)\right),\left(k_{1},\ldots,k_{n}\right)\right)&\mapsto\left(\varepsilon_{1}k_{\sigma\left(1\right)},\ldots,\varepsilon_{n}k_{\sigma\left(n\right)}\right)
\end{aligned}
\end{equation}
where $k_{i}\in\mathbb{Z}_{q-1}$ for $i=1,\ldots,n$. By this description
of the action of $W_{n}$ on $\widehat{T^{F}}$, we deduce that 
\begin{prop}
\label{prop:20}
Every $W_{n}$-orbit in $\widehat{T^{F}}$ can be represented
by a character 
\begin{equation}
\theta\sim\left(\stackrel{\lambda_{1}}{\overbrace{k_{1},\ldots,k_{1}}},\overset{\lambda_{2}}{\overbrace{k_{2},\ldots,k_{2}}},\ldots,\overset{\lambda_{l}}{\overbrace{k_{l},\ldots,k_{l}}},\overset{\alpha_{1}}{\overbrace{0,\ldots,0}},\overset{\alpha_{\epsilon}}{\overbrace{\frac{q-1}{2},\ldots,\frac{q-1}{2}}}\right)\label{eq:1.7}
\end{equation}
where $\lambda=\left(\lambda_{1}\geq\ldots\geq\lambda_{l}\right)\vdash c$ is a partition of a natural number $c\leq n$,
$\left|\lambda\right|+\alpha_{1}+\alpha_{\epsilon}=n$, $k_{i}\in Q:=\left\{ 1,\ldots,\frac{q-3}{2}\right\} $
and $k_{i}\neq k_{j}$ for any $i,j=1,\ldots,l$, and $k_{i}<k_{j}$
if $\lambda_{i}=\lambda_{j}$.
\end{prop}
There is an easy description of $S_{\theta}$ for a character $\theta$
of the form \ref{eq:1.7}: in fact, from how $W_{n}$ acts on $\widehat{T^{F}}$,
it is easy to see that 
\[
S_{\theta}=S_{\lambda,\alpha_{1},\alpha_{\epsilon}}:={\left(\prod\limits_{\substack{i=1}}^{l}{S_{\lambda_{i}}}\right)}\times\ W_{\alpha_{1}}\times W_{\alpha_{\epsilon}}.
\]
This a Coxeter group of type $A_{\lambda_{1}-1}\times\cdots\times A_{\lambda_{l}-1}\times B_{\alpha_{1}}\times B_{\alpha_{\epsilon}}$.
So let $\chi\in R_{T}^{G}{\left(\theta\right)}$. By Proposition \ref{prop:20},
we can take $\theta$ of the form \ref{eq:1.7}.
\begin{defn}
\label{defn:type}
If $\beta$ is the irreducible character of $S_{\theta}$ corresponding
to $\chi$, define the $4$-tuple 
\begin{equation}
\tau:=\left(\lambda,\alpha_{1},\alpha_{\epsilon},\beta\right)\label{eq:1.9}
\end{equation}
as the \emph{type} of $\chi$. If $\tau=\left(\lambda,\alpha_{1},\alpha_{\epsilon},\beta\right)$ and $\varepsilon$ is the sign character of $S_{\lambda,\alpha_{1}, \alpha_{\epsilon}}$, define the \emph{type dual to $\tau$} as 
\begin{equation}
\label{eq:1.10-1}
\tau^{\prime}:=\left(\lambda,\alpha_{1},\alpha_{\epsilon},\varepsilon\beta\right)
\end{equation}
\end{defn}
\noindent
$\mathbf{Notation}$: We write $\tau{\left(\chi\right)}$ for the type of a character $\chi$ 
and $\chi_{\tau}$ to denote a character of a fixed type $\tau$.
\begin{rem}
\label{rem:29}
It follows from the definition of type and Proposition \ref{prop:19} that $\tau{\left(1_{G^{F}}\right)}=\left(\hat{0},n,0,\beta \right)$, where $\beta(1)=1$. In fact, since $1_{G^{F}}\in R^{G}_{T}(1_{T^{F}})$, using Frobenius reciprocity and equation \ref{eq:1.5}, we get
\[
\beta{\left(1\right)}=\left\langle 1_{G^{F}},R_{T}^{G}{\left(1_{T^{F}}\right)}\right\rangle_{G^{F}}=\left\langle 1_{B^{F}},1_{B^{F}}\right\rangle_{B^{F}}=1.
\]
In this case, $B^{F}$ is the Borel subgroup of upper triangular matrices in $G^{F}$.
\end{rem}
\begin{rem}
\label{rem:22}
If $\chi_{1}$ and $\chi_{2}$ are irreducible constituents
of $R_{T}^{G}{\left(\theta_{1}\right)}$ and $R_{T}^{G}{\left(\theta_{2}\right)}$
respectively such that $\tau{\left(\chi_{1}\right)}=\tau{\left(\chi_{2}\right)}=\left(\lambda,\alpha_{1},\alpha_{\epsilon},\beta\right)$,
then by Proposition \ref{prop:19} we have that $\left\langle \chi_{1},R_{T}^{G}{\left(\theta_{1}\right)}\right\rangle _{G^{F}}=\left\langle \chi_{2},R_{T}^{G}{\left(\theta_{2}\right)}\right\rangle _{G^{F}}=\beta{\left(1\right)}$.
\end{rem}
\begin{prop}
\label{prop:25}
If $\chi\in R_{T}^{G}{\left(\theta\right)}$ with $\theta$
of the form \ref{eq:1.7}, then $\chi{\left(1\right)}$ only depends
on $\tau{\left(\chi\right)}=\left(\lambda,\alpha_{1},\alpha_{\epsilon},\beta\right)$.
\end{prop}
\begin{proof}
If $\alpha\in\Phi$, by \cite[Section §4]{key-9}, $q_{\alpha}{\left(\theta\right)}=1$
or $q_{\alpha}{\left(\theta\right)}=q$, so according to Proposition
\ref{prop:18}, equation \ref{eq:1.6} and Definition \ref{def:14}, it is
sufficient to prove that the set $\Gamma$ defined as in \ref{eq:1.5-1} only depends on the triple $\left(\lambda,\alpha_{1},\alpha_{\epsilon}\right)$.
Moreover, again by \cite[Section §4]{key-9}, $q_{\alpha}{\left(\theta\right)}=q_{\widetilde{\alpha}}{\left(\widetilde{\theta}\right)}$,
where $\widetilde{\alpha}$ and $\widetilde{\theta}$ are the restriction
of $\alpha$ and $\theta$ on the maximal torus $T_{\alpha}$ in $\left\langle X_{\alpha},X_{\alpha^{-1}}\right\rangle \cong\mathrm{SL}{\left(2,\mathbb{F}_{q}\right)}$,
and $q=q_{\widetilde{\alpha}}{\left(\widetilde{\theta}\right)}$ if
and only if $\widetilde{\theta}$ is the trivial character of $T_{\alpha}$. 

Let us give an explicit description of the root subgroups $X_{\alpha}$
in $\mathrm{Sp}{\left(2n,\mathbb{F}_{q}\right)}$ and of the corresponding $T_{\alpha}$'s: recall that $\Phi=\left\{ \epsilon_{i}^{\pm1}\epsilon_{j}^{\pm1}\mid1\leq i<j\leq n\right\} \cup\left\{ \epsilon_{i}^{\pm2}\mid1\leq i\leq n\right\} $,
with $\epsilon_{i}{\left(t\right)}:=t_{i}$ if $t=\mathrm{diag}{\left(t_{1},\ldots,t_{n},t_{n}^{-1},\ldots,t_{1}^{-1}\right)}$,
so
\[
\begin{aligned}
&X_{\epsilon_{i}\epsilon_{j}^{-1}}=\left\{ 1+t\left(E_{i,j}-E_{-j,-i}\right)\mid t\in\mathbb{F}_{q}\right\},\\
&X_{\epsilon_{i}^{-1}\epsilon_{j}}=\left\{ 1+t\left(E_{-i,-j}-E_{j,i}\right)\mid t\in\mathbb{F}_{q}\right\},\\
&X_{\epsilon_{i}\epsilon_{j}}=\left\{ 1+t\left(E_{i,-j}+E_{j,-i}\right)\mid t\in\mathbb{F}_{q}\right\},\\
&X_{\epsilon_{i}^{-1}\epsilon_{j}^{-1}}=\left\{ 1+t\left(E_{-i,j}+E_{-j,i}\right)\mid t\in\mathbb{F}_{q}\right\},\\
&X_{\epsilon_{i}^{2}}=\left\{ 1+tE_{i,-i}\mid t\in\mathbb{F}_{q}\right\},\\
&X_{\epsilon_{i}^{-2}}=\left\{ 1+tE_{-i,i}\mid t\in\mathbb{F}_{q}\right\};
\end{aligned}
\]
\[
\begin{aligned}
&T_{\epsilon_{i}\epsilon_{j}^{-1}}=\biggl\{
\begin{aligned}[t]
\mathrm{diag}\biggl(
&1,\ldots,1,\overset{i}{\overbrace{\lambda}},\ldots,\overset{j}{\overbrace{\lambda^{-1}}},1,\dots\\
&\ldots,1,\overset{-j}{\overbrace{\lambda}},\ldots,\overset{-i}{\overbrace{\lambda^{-1}}},1,\ldots,1\biggr)
\mid\lambda\in\mathbb{K}^{\times}\biggr\},
\end{aligned}\\
&T_{\epsilon_{i}\epsilon_{j}}=\biggl\{
\begin{aligned}[t]
\mathrm{diag}\biggl(
&1,\ldots,1,\overset{i}{\overbrace{\lambda}},\ldots,\overset{j}{\overbrace{\lambda}},1,\dots\\
&\ldots,1,\overset{-j}{\overbrace{\lambda^{-1}}},\ldots,\overset{-i}{\overbrace{\lambda^{-1}}},1,\ldots,1\biggr)\mid\lambda\in\mathbb{K}^{\times}\biggr\},
\end{aligned}\\
&T_{\epsilon_{i}^{2}}=\biggl\{ \mathrm{diag}\biggl(1,\ldots,1,\overset{i}{\overbrace{\lambda}},\ldots,\overset{-i}{\overbrace{\lambda^{-1}}},1,\ldots,1\biggr)\mid\lambda\in\mathbb{K}^{\times}\biggr\}.
\end{aligned}
\]
Remember that if $\theta\in\mathrm{Irr}{\left(T^{F}\right)}$, $\theta\sim\left(m_{1},\ldots,m_{n}\right)\in\mathbb{Z}_{q-1}^{n}$ and 
\[
t=\mathrm{diag}{\left(\lambda_{1},\ldots,\lambda_{n},\lambda_{n}^{-1},\ldots,\lambda_{1}^{-1}\right)}
\]
is an element of $T^{F}$, then
\begin{equation}
\begin{aligned}
\theta:T^{F}&\rightarrow\mathbb{C}^{\times}\\
t&\mapsto\theta{\left(t\right)}:=\prod\limits_{\substack{i=1}}^{n}{\lambda_{i}^{m_{i}}}.
\end{aligned}
\end{equation}
where we denote with the same symbol the image of $\lambda_{i}$ via a fixed homomorphism $f:\mathbb{F}_{q}^{\times}\hookrightarrow\mathbb{C}^{\times}$ for any $i=1,\dots,n$.
Therefore, it is easy to see that $\Gamma=A\cup B\cup C$, where 
\begin{equation}
\label{eq:descrGamma}
\begin{aligned}
&A:=\left\{\left(\epsilon_{i}\epsilon_{i+1}^{-1}\right)^{\pm1}\mid i\neq\lambda_{1},\lambda_{1}+\lambda_{2},\ldots,\sum\limits_{\substack{i=1}}^{l}{\lambda_{i}},\sum\limits_{\substack{i=1}}^{l}{\lambda_{i}+\mu_{1}}\right\},\\
&B:=\left\{ \epsilon_{i}^{\pm2}\mid m+1\leq i\leq m+\alpha_{1}\right\},\\
&C:=\left\{\left(\epsilon_{i}\epsilon_{i+1}\right)^{\pm1}\mid m+1\leq i\leq n-1,\,i\neq m+\alpha_{1}\right\} 
\end{aligned}
\end{equation}
and the claim follows. 
\end{proof}

\begin{rem}
\label{rem:26}
It is easy to see from the definition of $W_{{S}_{\theta}}$ and \ref{eq:descrGamma}
that $W_{S_{\theta}}=\left(\prod\limits_{\substack{i=1}}^{l}{S_{\lambda_{i}}}\right)\times W_{\alpha_{1}}\times\left(S_{\alpha_{\epsilon}}\ltimes V_{\alpha_{\epsilon}}\right)$ where
\[
V_{\alpha_{\epsilon}}:=\left\{ \left(x_{1},\ldots,x_{\alpha_{\epsilon}}\right)\in\boldsymbol{\mu}_{\boldsymbol{2}}^{\alpha_{\epsilon}}\mid\prod\limits_{\substack{i=1}}^{\alpha_{\epsilon}}{x_{i}}=1\right\}
\]
so $W_{S_{\theta}}$ is a Coxeter group of type $A_{\lambda_{1}-1}\times\cdots\times A_{\lambda_{l}-1}\times B_{\alpha_{1}}\times D_{\alpha_{\epsilon}}$
of index smaller than $2$ in $S_{\theta}$. It follows that
\begin{equation}
\label{eq:1.100}
P_{\theta}{\left(q\right)}=\left(\prod_{i=1}^{l}P_{\lambda_{i}}{\left(q\right)}\right){P_{\alpha_{1}}{\left(q\right)}P_{\alpha_{\epsilon}}{\left(q\right)}}
\end{equation}
where $P_{\lambda_{i}}$, $P_{\alpha_{1}}$ and $P_{\alpha_{\epsilon}}$ are Poincaré polynomials of type $A_{\lambda_{i}-1}$ for $i=1,\dots,l$, $B_{\alpha_{1}}$ and $D_{\alpha_{\epsilon}}$ respectively.
\end{rem}

\begin{rem}
\label{rem:27}
Let $\chi_{\tau}$ be such that $\tau=\left(\lambda,\alpha_{1},\alpha_{\epsilon},\beta\right)$
and $\varphi$ one of the (at most two for Remark \ref{rem:26})
irreducible components of $\mathrm{Res}_{W_{S_{\theta}}}^{S_{\theta}}\left(\beta\right)$.
Then 
\[
\varphi=\left(\bigotimes\limits_{\substack{i=1}}^{l}{\varphi_{i}}\right)\otimes\varphi_{\alpha_{1}}\otimes\varphi_{\alpha_{\epsilon}}
\]
with $\varphi_{i}\in\mathrm{Irr}{\left(S_{\lambda_{i}}\right)}$ for
$i=1,\ldots,l$, $\varphi_{\alpha_{1}}\in\mathrm{Irr}{\left(W_{\alpha_{1}}\right)}$
and $\varphi_{\alpha_{\epsilon}}\in\mathrm{Irr}{\left(S_{\alpha_{\epsilon}}\ltimes V_{\alpha_{\epsilon}}\right)}$.
It follows from \ref{eq:1.6-2}, \ref{eq:1.100} and Remark \ref{rem:20} and \ref{rem:21} that 
\begin{equation}
\chi_{\tau}{\left(1\right)}=\frac{P{\left(q\right)}}{2}\left(\prod_{i=1}^{l}{\frac{d_{\varphi_{i}}{\left(q\right)}}{P_{\lambda_{i}}{\left(q\right)}}}\right)\frac{d_{\varphi_{\alpha_{1}}}{\left(q\right)}}{P_{\alpha_{1}}{\left(q\right)}}\frac{d_{\varphi_{\alpha_{\epsilon}}}{\left(q\right)}}{P_{\alpha_{\epsilon}}{\left(q\right)}}\label{eq:1.10}
\end{equation}
if $\varphi=\mathrm{Res}_{W_{S_{\theta}}}^{S_{\theta}}{\left(\beta\right)}$ and $\alpha_{\epsilon}\neq0$,
and 
\begin{equation}
\chi_{\tau}{\left(1\right)}=P{\left(q\right)}\left(\prod_{i=1}^{l}{\frac{d_{\varphi_{i}}{\left(q\right)}}{P_{\lambda_{i}}{\left(q\right)}}}\right)\frac{d_{\varphi_{\alpha_{1}}}{\left(q\right)}}{P_{\alpha_{1}}{\left(q\right)}}\frac{d_{\varphi_{\alpha_{\epsilon}}}{\left(q\right)}}{P_{\alpha_{\epsilon}}{\left(q\right)}}\label{eq:1.11}
\end{equation}
otherwise. Notice that, in this case, $P{\left(q\right)}$ is a Poincaré polynomial of type $B_{n}$. Since $\chi_{\tau}{\left(1\right)}$ divides $\left|\mathrm{Sp}{\left(2n,\mathbb{F}_{q}\right)}\right|$
as an integer number for any possible value of $q$, we deduce from Remark \ref{rem:21-1} that $\frac{\left|\mathrm{Sp}{\left(2n,\mathbb{F}_{q}\right)}\right|}{\chi_{\tau}{\left(1\right)}}\in\mathbb{Z}[q]$.
\end{rem}
\begin{rem}
\label{rem:28}
For odd $q\geq3$, since $\mathrm{Sp}\left(2n,\mathbb{F}_{q}\right)$ is a perfect group, unless $n=1$ and $q=3$, and Poincaré polynomials of Coxeter groups are always monic, by \ref{eq:1.10}, \ref{eq:1.11}, Remark \ref{rem:21-1} and by looking at the character table of $\mathrm{SL}(2,\mathbb{F}_{q})$ in \cite{key-12}, we have that $\chi_{\tau}{\left(1\right)}$ does not depend on $q$ if and only if $\chi_{\tau}=1_{\mathrm{Sp}\left(2n,\mathbb{F}_{q}\right)}$.
\end{rem}
We conclude this section defining the polynomial 
\begin{equation}
H_{\tau}{\left(q\right)}:=\frac{\left|\mathrm{Sp}{\left(2n,\mathbb{F}_{q}\right)}\right|}{\chi_{\tau}{\left(1\right)}}\in\mathbb{Z}[q]\label{eq:1.12}
\end{equation}
for all types $\tau$. By Remark \ref{rem:27}, one can use the formulas
contained in \cite[Theorem 10.5.2, 10.5.3]{key-11}. and the ones in \cite{key-16} or in \cite[10.5.1]{key-11} for
the Poincaré polynomial of a Coxeter group of type $A_{n-1}$, $B_{n}$ and $D_{n}$ to compute explicitly
the polynomials $H_{\tau}{\left(q\right)}$.
\begin{rem}
\label{rem:29}
Using \cite[Lemma 11.3.2]{key-7} and after some little algebra, one can prove that 
\begin{equation}
\label{eq:1.13}
H_{\tau}{\left(q^{-1}\right)}=\frac{\left(-1\right)^n}{q^{n\left(2n+1\right)}}{H_{\tau^{\prime}}}{\left(q\right)}.
\end{equation}
\end{rem}

\section{Parabolic $\mathrm{Sp}_{2n}$-character varieties}
\label{section:3}

In this section, we consider the standard presentation of the symplectic
group. So, if $\mathbb{K}$ is an algebraically closed field, $n$
a positive integer, then 
\[
\mathrm{Sp}{\left(2n,\mathbb{K}\right)}:=\left\{ A\in\mathfrak{gl}{\left(2n,\mathbb{K}\right)}\mid A^{t}JA=A\right\} 
\]
with $J:=\left(\begin{array}{cc}
0 & I_{n}\\
-I_{n} & 0
\end{array}\right)$.

\subsection{Geometry of character varieties}

Let $g\geq0$, $n>0$ be integers. Let $\mathbb{K}$
be an algebraically closed field with $\mathrm{char}\left(\mathbb{K}\right)\neq2$,
possessing a primitive $m$-th root of unity $\varphi$, for which
there exist natural numbers $m_{1},\ldots,m_{n}$ such that $\varphi^{m_{1}},\ldots,\varphi^{m_{n}}$
satisfy the following non-equalities, for every disjoint sets of indices
$J$ and $L$, not simultaneously empty, and every index $i$ in $\left\{ 1,\ldots n\right\} $:
\begin{equation}
\begin{aligned}
\prod_{j\in J}{\varphi^{m_{j}}}&\neq\prod_{l\in L}{\varphi^{m_{l}}},\\
\varphi^{2m_{i}}&\neq1.\label{eq:2.1}
\end{aligned}
\end{equation}

\begin{rem}
\label{rem:2.1}
Specializing \ref{eq:2.1} for $J=\left\{ j\right\} $
and $L=\left\{ l\right\} $ or $J=\left\{ j,l\right\} $ and $L=\emptyset$,
we have that $\varphi^{m_{j}}\neq\varphi^{m_{l}}$ and $\varphi^{m_{j}}\neq\varphi^{-m_{l}}$
respectively, so $\varphi^{\pm m_{1}},\ldots,\varphi^{\pm m_{n}}$
have to be all different. In particular, $m>n$.
\end{rem}
\begin{rem}
\label{rem:sub}
It is easy to see that the elements of any subset of $\left\{\varphi^{m_{1}},\ldots,\varphi^{m_{n}}\right\}$
satisfy conditions \ref{eq:2.1}. 
\end{rem}
\begin{example}
If $\varphi$ is a primitive $\left(2^{n}+1\right)$-th root of unity,
then it is easy to see that $\varphi,\varphi^{2},\ldots,\varphi^{2^{n-1}}$
satisfy \ref{eq:2.1}.
\end{example}
Consider the following algebraic variety over $\mathbb{K}$:
\begin{equation}
\mathcal{U}_{n}^{\xi}:=\left\{ \left(A_{1},B_{1},\ldots,A_{g},B_{g}\right)\in\mathrm{Sp}{\left(2n,\mathbb{K}\right)}^{2g}\mid\prod_{i=1}^{g}{\left[A_{i}:B_{i}\right]}=\xi\right\} =\mu^{-1}{\left(\xi\right)}\label{eq:2.2}
\end{equation}

where $\mu:\mathrm{Sp}{\left(2n,\mathbb{K}\right)}^{2g}\rightarrow\mathrm{Sp}{\left(2n,\mathbb{K}\right)}$
is given by
\begin{equation}
\mu{\left(A_{1},B_{1},\ldots,A_{g},B_{g}\right)}:=\prod_{i=1}^{g}{\left[A_{i}:B_{i}\right]}\label{eq:2.3}
\end{equation}
and $\xi=\mathrm{diag}{\left(\varphi^{m_{1}},\ldots,\varphi^{m_{n}},\varphi^{-m_{1}},\ldots,\varphi^{-m_{n}}\right)}$.
If $n=0$, we will assume that $\mathcal{U}_{n}^{\xi}=\left\{ \star\right\}$.
By Remark \ref{rem:2.1}, the centralizer of $\xi$ in $\mathrm{Sp}{\left(2n,\mathbb{K}\right)}$
is the maximal torus
\[
T=\left\{ \mathrm{diag}{\left(\lambda_{1},\ldots,\lambda_{n},\lambda_{1}^{-1},\ldots,\lambda_{n}^{-1}\right)}\mid\lambda_{i}\in\mathbb{\mathbb{K}^{\star}},\,i=1,\ldots,n\right\} 
\]
and acts by conjugation on $\mathcal{U}_{n}^{\xi}$:
\begin{equation}
\begin{aligned}
\sigma:T\times\mathcal{U}_{n}^{\xi}&\rightarrow\mathcal{U}_{n}^{\xi}\\
\left(h,\left(A_{1},B_{1},\ldots,A_{g},B_{g}\right)\right)&\mapsto\left(h^{-1}A_{1}h,h^{-1}B_{1}h,\ldots,h^{-1}A_{g}h,h^{-1}B_{g}h\right).
\end{aligned}
\end{equation}

As the center $\boldsymbol{Z}=\left\{\pm I_{2n}\right\}\leq T$ of $\mathrm{Sp}{\left(2n,\mathbb{K}\right)}$ 
acts trivially, this action induces an action
\[
\bar{\sigma}:T/\boldsymbol{Z}\times\mathcal{U}_{n}^{\xi}\rightarrow\mathcal{U}_{n}^{\xi}.
\]
In the following, if $X\in\mathcal{U}_{n}^{\xi}$, $h\in T$, we will
write $h^{-1}Xh$ instead of $\sigma{\left(h,X\right)}$ or $\overline{\sigma}{\left(h,X\right)}$.
\begin{prop}
\label{prop:2.3}
Let $X=\left(A_{1},B_{1},\ldots,A_{g},B_{g}\right)$
be an element of $\mathcal{U}_{n}^{\xi}$. Let $\boldsymbol{\mu}_{\boldsymbol{2}}^{n}$ be the
subgroup of the involution matrices in $T$. Then $T_{X}:=\mathrm{Stab}_{\sigma}{\left(X\right)}\leq\boldsymbol{\mu}_{\boldsymbol{2}}^{n}$.
\end{prop}
\begin{proof}
Let $Z$ be an element of $T_{X}$. Then $ZA_{i}=A_{i}Z$, $ZB_{i}=B_{i}Z$
for any $i=1,\ldots,g$. Suppose that $Z\notin\boldsymbol{\mu}_{\boldsymbol{2}}^{n}$. Then $Z$ has
an eigenvalue $\alpha$ different from $\pm1$. Permute the eigenvalues
of $Z$ in order to collect them in groups such that all the elements
in the same group are equal. By a further permutation, we can assume
that the first group of eigenvalues of $Z$ is made by the $\alpha$'s.
This is equivalent to the action of a permutation matrix $\pi$ by
conjugation on $Z$. Denote by $\pi{\left(\cdot\right)}$ the conjugation
by $\pi$. Then 
\begin{equation}
\label{eq:2.5}
\begin{aligned}
\pi{\left(Z\right)}\pi{\left(A_{i}\right)}&=\pi{\left(A_{i}\right)}\pi{\left(Z\right)}\\
\pi{\left(Z\right)}\pi{\left(B_{i}\right)}&=\pi{\left(B_{i}\right)}\pi{\left(Z\right)}
\end{aligned}
\end{equation}
for all $i=1,\ldots,g$, and
\begin{equation}
\prod_{i=1}^{g}{\left[\pi{\left(A_{i}\right)}:\pi{\left(B_{i}\right)}\right]}=\pi{\left(\xi\right)}.\label{eq:2.7}
\end{equation}
Now, by \ref{eq:2.5}, we have that $\pi{\left(A_{i}\right)}=\mathrm{diag}{\left(A_{i}^{1},\ldots,A_{i}^{k}\right)}$,
$\pi{\left(B_{i}\right)}=\mathrm{diag}{\left(B_{i}^{1},\ldots,B_{i}^{k}\right)}$
for any $i=1,\ldots,g$, where $k$ is the number of different eigenvalues
of $Z$ and the $A_{i}^{h}$'s and $B_{i}^{h}$ are square matrices
whose sizes are equal to the multiplicity of the $h$-th eigenvalue
of $Z$, $h=1,\ldots,k$. From \ref{eq:2.7}, writing $\pi{\left(\xi\right)}=\mathrm{diag}{\left(D^{1},\ldots,D^{k}\right)}$,
follows that 
\[
\mathrm{diag}{\left(\prod_{i=1}^{g}{\left[A_{i}^{1}:B_{i}^{1}\right]},\ldots,\prod_{i=1}^{g}{\left[A_{i}^{k}:B_{i}^{k}\right]}\right)}=\mathrm{diag}{\left(D^{1},\ldots,D^{k}\right)}.
\]
As the determinant of a commutator is 1, the determinant
of $D^{h}$ has to be equal to 1 for any $h=1,\ldots,k$. In particular,
$\mathrm{det}{\left(D^{1}\right)}=1$. But since $\alpha\neq\alpha^{-1}$,
there is a $\varphi^{m_{j}}$, for some $j$, that is an eigenvalue
of $D^{1}$, but not $\varphi^{-m_{j}}$, so $\mathrm{det}\left(D^{1}\right)$
cannot be equal to $1$, because $\varphi^{m_{1}},\ldots,\varphi^{m_{n}}$
satisfy the inequalities \ref{eq:2.1}, and this is a contradiction.
\end{proof}
\begin{defn}
\label{def:2.4}Fix a subgroup $H$ of $\boldsymbol{\mu}_{\boldsymbol{2}}^{n}$ containing
$\boldsymbol{Z}$. Define the following subsets of $\mathcal{U}_{n}^{\xi}$:
\begin{equation}
\widetilde{\mathcal{U}}{}_{n,H}^{\xi}:=\left\{ X\in\mathcal{U}_{n}^{\xi}\mid H=T_{X}\right\} \label{eq:2.8}
\end{equation}
\begin{equation}
\,\,\,\mathcal{U}_{n,H}^{\xi}:=\left\{ X\in\mathcal{U}_{n}^{\xi}\mid H\subseteq T_{X}\right\}.
\label{eq:2.9}
\end{equation}
\end{defn}
\begin{rem}
\label{rem:2.5}
It is evident from the previous Definition
\ref{def:2.4} that $\widetilde{\mathcal{U}}_{n,H}^{\xi}$ is an open
subset of the closed affine variety $\mathcal{U}_{n,H}^{\xi}$. In
particular, $\left\{ \widetilde{\mathcal{U}}_{n,H}^{\xi}\right\} _{\boldsymbol{Z}\leq H\leq\boldsymbol{\mu}_{\boldsymbol{2}}^{n}}$
is a stratification of $\mathcal{U}_{n}^{\xi}$. Moreover, $\mathcal{U}_{n,\boldsymbol{\mu_{2}}}^{\xi}=\mathcal{U}_{n}^{\xi}$,
so $\widetilde{\mathcal{U}}_{n,\boldsymbol{\mu_{2}}}^{\xi}$ is an open subset
of $\mathcal{U}_{n}^{\xi}$.
\end{rem}
\begin{prop}
\label{prop:2.4}$\widetilde{\mathcal{U}}_{n,H}^{\xi}$
and $\mathcal{U}_{n,H}^{\xi}$ are stable under the action $\sigma$
of $T$. 
\end{prop}
\begin{proof}
Let $X=\left(A_{1},B_{1},\ldots,A_{g},B_{g}\right)$ be
an element of $\mathcal{U}_{n,H}^{\xi}$, $Z\in H$. Then $A_{i}Z=ZA_{i}$, $B_{i}Z=ZB_{i}$ for any $i=1,\ldots,g$. It follows that, if $\omega\in T$, $X\in\mathcal{U}_{n,H}^{\xi}$
if and only if $\omega^{-1}X\omega\in\mathcal{U}_{n,\omega^{-1}H\omega}^{\xi}$. But since
$T$ is abelian, $\omega^{-1}H\omega=H$ and the assertion follows. The proof for $\widetilde{\mathcal{U}}_{n,H}^{\xi}$ is completely analogous.
\end{proof}
\begin{rem}
\label{rem:2.6}
By its definition, it is easy to see that $\widetilde{\mathcal{U}}_{n,H}^{\xi}$ admits a finite open cover $\left\{ \widetilde{\mathcal{U}}_{i,H}\right\}_{i\in I}$ of affine $T$-stable subsets. For more details, see \cite[Remark 3.1.10]{key-phd}.
\end{rem}
\begin{defn}
\label{def:2.5}A \emph{parabolic $\mathrm{Sp}{\left(2n,\mathbb{K}\right)}$-character
variety} of a closed Riemann surface of genus $g$ is the categorical
quotient 
\begin{equation}
\mathcal{M}_{n}^{\xi}:=\mathcal{U}_{n}^{\xi}//T=\mathrm{Spec}{\left(\mathbb{K}{\left[\mathcal{U}_{n}^{\xi}\right]}^{T}\right)}.\label{eq:2.10}
\end{equation}
More generally, define the categorical quotient
\begin{equation}
\mathcal{M}_{n,H}^{\xi}:=\mathcal{U}_{n,H}^{\xi}//T=\mathrm{Spec}{\left(\mathbb{K}{\left[\mathcal{U}_{n,H}^{\xi}\right]}^{T}\right)}.\label{eq:2.11}
\end{equation}
\end{defn}
\begin{rem}
\label{rem:2.8}Since $H$ acts trivially on $\mathcal{U}_{n,H}^{\xi}$,
we can define $\mathcal{M}_{n,H}^{\xi}$ as the categorical quotient
$\mathcal{U}_{n,H}^{\xi}//\left(T/H\right)$.
\end{rem}
\begin{prop}
\label{prop:2.7}
$\mathcal{M}_{n,H}^{\xi}$ is a geometric
quotients for any $\boldsymbol{Z}\leq H\leq\boldsymbol{\mu}_{\boldsymbol{2}}^{n}$. 
\end{prop}
\begin{proof}
Since $\mathcal{M}_{n,H}^{\xi}$ is a categorical quotient
of an affine variety by the action of an affine reductive algebraic
group, it is a good quotient (for a definition of a good quotient see \cite[Definition 2.36]{key-Ho}), so by \cite[Corollary 2.39 ii)]{key-Ho}, it is sufficient to prove
that all the orbits are closed. By \ref{prop:2.3}, for every $X\in\mathcal{U}_{n,H}^{\xi}$,
$\mathrm{dim}{\left(T_{X}\right)}=0$. It follows, denoting the orbit of $X$ by $TX$, that $\mathrm{dim}{\left(TX\right)}=\mathrm{dim}{\left(T\right)}$
from the Orbit-Stabiliser Theorem. 

Now, suppose that there exists
a non closed orbit. Then, by \cite[Proposition in 8.3]{key-22}, its boundary is not empty and it is
a union of orbits of strictly smaller dimension. But this contradicts
the fact that all the orbits have the same dimension. 
\end{proof}
By Proposition \ref{prop:2.4}, together with the properties
of geometric quotients, we can give the following
\begin{defn}
\label{defn:stratum}
For every $\boldsymbol{Z}\leq H\leq\boldsymbol{\mu}_{\boldsymbol{2}}^{n}$, define the geometric quotient
\[
\widetilde{\mathcal{M}}_{n,H}^{\xi}:=\widetilde{\mathcal{U}}_{n,H}^{\xi}/T.
\]
\end{defn}
\begin{rem}
Since $\mathcal{M}_{n}$ is a geometric quotient because
of Proposition \ref{prop:2.7}, it has the quotient topology, hence,
by Proposition \ref{prop:2.4} and \ref{rem:2.5}, $\mathcal{M}_{n,H}^{\xi}$
is a closed affine variety and $\widetilde{\mathcal{M}}_{n,H}^{\xi}$
is an open subset of it. In particular, $\left\{ \widetilde{\mathcal{M}}_{n,H}^{\xi}\right\} _{\boldsymbol{Z}\leq H\leq\boldsymbol{\mu}_{\boldsymbol{2}}^{n}}$
is a stratification of $\mathcal{M}_{n}^{g}$ and $\widetilde{\mathcal{M}}_{n,\boldsymbol{\mu_{2}}}^{\xi}$
is an open subset of the character variety $\mathcal{M}_{n}^{\xi}$.
\end{rem}

\begin{rem}
As in \ref{rem:2.8}, we can realize $\widetilde{\mathcal{M}}_{n,H}^{\xi}$
as the geometric quotient of $\widetilde{\mathcal{U}}_{n,H}^{\xi}$
by the \emph{free }action of the affine algebraic group $T/H$.
\end{rem}
\begin{rem}
\label{rem:2.9}
Thanks to \ref{rem:2.6}, we get a finite open affine cover of $\widetilde{\mathcal{M}}_{n,H}^{\xi}$
given by $\left\{ \widetilde{\mathcal{M}}_{i,H}\right\}_{i\in I}$, where $\widetilde{\mathcal{M}}_{i,H}=\widetilde{\mathcal{U}}_{i,H}/T$.
\end{rem} 
\begin{prop}
\label{prop:2.13}
The variety $\mathcal{U}_{n}^{\xi}$ is
non singular and equidimensional. The dimension of each connected
component of $\mathcal{U}_{n}^{\xi}$ is given by 
\begin{equation}
\label{eq:dim}
\mathrm{dim}{\left(\mathcal{U}_{n}^{\xi}\right)}={\left(2g-1\right)}n{\left(2n+1\right)}.
\end{equation}
\end{prop}
\begin{proof}
We follow the strategy of \cite[Theorem 2.2.5]{key-5}, with slight
variations. Assume that $g>0$. It is enough to show that at a
solution $s=\left(A_{1},B_{1},\ldots,A_{g},B_{g}\right)\in\mathrm{Sp}{\left(2n,\mathbb{K}\right)}^{2g}$
of the equation
\begin{equation}
\left[A_{1}:B_{1}\right]\cdots\left[A_{g}:B_{g}\right]=\xi\label{eq:2.12}
\end{equation}
the derivative of $\mu$ on the tangent spaces 
\[
d\mu_{s}:T_{s}{\left(\mathrm{Sp}{\left(2n,\mathbb{K}\right)}\right)}^{2g}\rightarrow T_{\xi}{\left(\mathrm{Sp}{\left(2n,\mathbb{K}\right)}\right)}
\]
is surjective. So take $\left(X_{1},Y_{1},\ldots,X_{g},Y_{g}\right)\in T_{s}{\left(\mathrm{Sp}{\left(2n,\mathbb{K}\right)}\right)}^{2g}$.
Then differentiate $\mu$ to get: 

$d\mu_{s}{\left(X_{1},Y_{1},\ldots,X_{g},Y_{g}\right)}=$
\begin{equation}
\begin {aligned}
&\sum\limits_{\substack{i=1}}^{g}{{\left[A_{1}:B_{1}\right]}{\cdots}{\left[A_{i-1}:B_{i-1}\right]}{X_{i}B_{i}A_{i}^{-1}B_{i}^{-1}}{\left[A_{i+1}:B_{i+1}\right]}{\cdots}{\left[A_{g}:B_{g}\right]}}\\
+&\sum\limits_{\substack{i=1}}^{g}{{\left[A_{1}:B_{1}\right]}{\cdots}{\left[A_{i-1}:B_{i-1}\right]}{A_{i}Y_{i}A_{i}^{-1}B_{i}^{-1}}{\left[A_{i+1}:B_{i+1}\right]}{\cdots}{\left[A_{g}:B_{g}\right]}}\\
-&\sum\limits_{\substack{i=1}}^{g}{{\left[A_{1}:B_{1}\right]}{\cdots}{\left[A_{i-1}:B_{i-1}\right]}{A_{i}B_{i}A_{i}^{-1}X_{i}A_{i}^{-1}B_{i}^{-1}}{\left[A_{i+1}:B_{i+1}\right]}{\cdots}{\left[A_{g}:B_{g}\right]}}\\
-&\sum\limits_{\substack{i=1}}^{g}{{\left[A_{1}:B_{1}\right]}{\cdots}{\left[A_{i-1}:B_{i-1}\right]}{A_{i}B_{i}A_{i}^{-1}B_{i}^{-1}Y_{i}B_{i}^{-1}}{\left[A_{i+1}:B_{i+1}\right]}{\cdots}{\left[A_{g}:B_{g}\right]}}
\end{aligned}
\end{equation}
and using \ref{eq:2.12}, for each of the four terms, we get:
\begin{equation}
\label{eq:2.13}
d\mu_{s}{\left(X_{1},Y_{1},\ldots,X_{g},Y_{g}\right)}=\sum\limits_{\substack{i=1}}^{g}{f_{i}{\left(X_{i}\right)}+g_{i}{\left(Y_{i}\right)}},
\end{equation}
where we define linear maps 
\[
\begin{aligned}
f_{i}:T_{A_{i}}{\left(\mathrm{Sp}{\left(2n,\mathbb{K}\right)}\right)}&\rightarrow\mathfrak{gl}{\left(2n,\mathbb{K}\right)}\\
g_{i}:T_{B_{i}}{\left(\mathrm{Sp}{\left(2n,\mathbb{K}\right)}\right)}&\rightarrow\mathfrak{gl}{\left(2n,\mathbb{K}\right)}
\end{aligned}
\]
by $f_{i}{\left(X\right)}:=$
\[
\prod\limits_{\substack{j=1}}^{i-1}{\left[A_{j}:B_{j}\right]}{\left(XA_{i}^{-1}-A_{i}B_{i}A_{i}^{-1}XB_{i}^{-1}A_{i}^{-1}\right)}{\prod\limits_{\substack{j=1}}^{i-1}{\left[B_{i-j}:A_{i-j}\right]}\xi}
\]
and 
$g_{i}{\left(Y\right)}:=$
\[
{\prod\limits_{\substack{j=1}}^{i-1}{\left[A_{j}:B_{j}\right]}}{\left(A_{i}YB_{i}^{-1}A_{i}^{-1}-A_{i}B_{i}A_{i}^{-1}B_{i}^{-1}YA_{i}B_{i}^{-1}A_{i}^{-1}\right)}{\prod\limits_{\substack{j=1}}^{i-1}{\left[B_{i-j}:A_{i-j}\right]}\xi}.
\]
We claim that $f_{i}$ and $g_{i}$ take values in $T_{\xi}{\left(\mathrm{Sp}{\left(2n,\mathbb{K}\right)}\right)}$.
We will prove it only for $f_{i}$, the proof for $g_{i}$ being completely
analogous. 

What we have to prove is that $B=f_{i}{\left(X\right)}\xi^{-1}$
is a hamiltonian matrix for every $X\in T_{A_{i}}{\left(\mathrm{Sp}{\left(2n,\mathbb{K}\right)}\right)}$,
i.e., $B^{t}J=-JB$. 
Call 
\[
\begin{aligned}
U&=\prod\limits_{\substack{j=1}}^{i-1}{\left[A_{j}:B_{j}\right]},\\
V&=XA_{i}^{-1}-A_{i}B_{i}A_{i}^{-1}XB_{i}^{-1}A_{i}^{-1}.
\end{aligned}
\]
Notice that, since $A_{j}$ and $B_{j}$ are symplectic for every
$j=1,\ldots,i$ and $X\in T_{A_{i}}{\left(\mathrm{Sp}{\left(2n,\mathbb{K}\right)}\right)}$,
$U$ and $U^{-1}$ are symplectic. Moreover, the facts that $A_{i}^{-1}X$ is hamiltonian and $A_{i}B_{i}$ is symplectic imply that $A_{i}B_{i}A_{i}^{-1}XB_{i}^{-1}A_{i}^{-1}$ is hamiltonian and since $XA_{i}^{-1}$ is hamiltonian, $V$ and $V^{t}$ are hamiltonian too. Then 
\[
\begin{aligned}
B^{t}J&=\left(U^{-1}\right)^{t}V^{t}U^{t}J=\left(U^{-1}\right)^{t}V^{t}JU^{-1}\\
&=-\left(U^{-1}\right)^{t}JVU^{-1}=-JUVU^{-1}=-JB
\end{aligned}
\]
that is our claim. 

Assume that $Z^{\prime}\in T_{\xi}{\left(\mathrm{Sp}{\left(2n,\mathbb{K}\right)}\right)}$
such that
\begin{equation}
\label{eq:2.14}
\mathrm{Tr}{\left(JZ^{\prime}J^{-1}d\mu_{s}{\left(X_{1},Y_{1},\ldots,X_{g},Y_{g}\right)}\right)}=0.
\end{equation}
By \ref{eq:2.13}, this is equivalent to 
\[
\mathrm{Tr}{\left(JZ^{\prime}J^{-1}f_{i}{\left(X_{i}\right)}\right)}=\mathrm{Tr}{\left(JZ^{\prime}J^{-1}g_{i}{\left(Y_{i}\right)}\right)}=0
\]
for all $i$ and $X_{i}\in T_{A_{i}}{\left(\mathrm{Sp}{\left(2n,\mathbb{K}\right)}\right)}$,
$Y_{i}\in T_{B_{i}}{\left(\mathrm{Sp}{\left(2n,\mathbb{K}\right)}\right)}$.
We show by induction on $i$ that this implies that, if $Z^{\prime}=\xi Z$,
with $Z$ hamiltonian, $C:=JZJ^{-1}$ commutes with $A_{i}$ and $B_{i}$.
Notice that $C$ is hamiltonian. Assume we have already proved this
for $j<i$ and calculate
\[
\begin{aligned}
0&=\mathrm{Tr}{\left(JZ^{\prime}J^{-1}f_{i}{\left(X_{i}\right)}\right)}\\
&=\mathrm{Tr}{\left(C\left(X_{i}A_{i}^{-1}-A_{i}B_{i}A_{i}^{-1}X_{i}B_{i}^{-1}A_{i}^{-1}\right)\right)}\\
&=\mathrm{Tr}{\left(\left(A_{i}^{-1}CA_{i}-B_{i}A_{i}^{-1}CA_{i}B_{i}\right)A_{i}^{-1}X_{i}\right)}
\end{aligned}
\]
for all $X_{i}\in T_{A_{i}}{\left(\mathrm{Sp}{\left(2n,\mathbb{K}\right)}\right)}$.
Since $A_{i}^{-1}CA_{i}-B_{i}A_{i}^{-1}CA_{i}B_{i}$ and $A_{i}^{-1}X_{i}$
are hamiltonian, and $\mathrm{Tr}{\left(\cdot,\cdot\right)}$ is a non
degenerate symmetric bilinear form over $\mathfrak{sp}{\left(2n,\mathbb{K}\right)}$,
when $\mathrm{char}{\left(\mathbb{K}\right)}\neq2$, $C$ commutes with
$A_{i}B_{i}A_{i}^{-1}$. Similarly we have
\[
\begin{aligned}
0&=\mathrm{Tr}{\left(JZ^{\prime}J^{-1}g_{i}{\left(Y_{i}\right)}\right)}\\
&=\mathrm{Tr}{\left(\left(B_{i}^{-1}A_{i}^{-1}CA_{i}B_{i}-A_{i}^{-1}B_{i}^{-1}A_{i}^{-1}CA_{i}B_{i}A_{i}^{-1}\right)B_{i}^{-1}Y_{i}\right)}
\end{aligned}
\]
which implies that $C$ commutes with $A_{i}B_{i}A_{i}B_{i}^{-1}A_{i}^{-1}$.
Thus $C$ commutes with $A_{i}$ and $B_{i}$, hence with $\xi=\prod\limits_{\substack{i=1}}^{g}{\left[A_{i}:B_{i}\right]}$.
It follows that 
\[
C=\mathrm{diag}{\left(\lambda_{1},\ldots,\lambda_{n},-\lambda_{1},\ldots,-\lambda_{n}\right)}.
\]
Arguing as in Proposition \ref{prop:2.3}, we can prove by contradiction
that $C=0$. Thus there is no non-zero $Z^{\prime}$ such that \ref{eq:2.14}
holds for all $X_{i}$ and $Y_{i}$. Since $\varphi{\left(A,B\right)}:=\mathrm{Tr}{\left(JAJ^{-1}B\right)}$
is symmetric non degenerate bilinear form over $T_{\xi}{\left(\mathrm{Sp}{\left(2n,\mathbb{K}\right)}\right)}$
when $\mathrm{char}{\left(\mathbb{K}\right)}\neq2$, this implies that
$d\mu$ is surjective at any solution $s$ of \ref{eq:2.12}. Thus
$\mathcal{U}_{n}^{\xi}$ is non singular and equidimensional. Finally,
we see that the dimension of (each connected component of) $\mathcal{U}_{n}^{\xi}$
is
\[
\mathrm{dim}{\left(\mathrm{Sp}{\left(2n,\mathbb{K}\right)}^{2g}\right)}-\mathrm{dim}{\left(\mathrm{Sp}{\left(2n,\mathbb{K}\right)}\right)}={\left(2g-1\right)}n{\left(2n+1\right)}
\]
proving the second claim.
\end{proof}

\begin{cor}
\label{cor:2.14}
The dimension of (each connected component of) $\mathcal{M}_{n}^{\xi}$ is equal
to $d_{n}:={\left(2g-1\right)}n{\left(2n+1\right)}-n$.
\end{cor}
\begin{proof}
By Proposition \ref{prop:2.3} and \ref{prop:2.7}, we have that 
\[
\mathrm{dim}{\left(\mathcal{M}_{n}^{\xi}\right)}=\mathrm{dim}{\left(\mathcal{U}_{n}^{\xi}\right)}-\mathrm{dim}{\left(T\right)}
\]
so the claim easily follows from Proposition \ref{prop:2.13}.
\end{proof}

\subsection{Geometry of $\mathcal{U}_{n,H}^{\xi}$ }

The goal of this section is to describe the geometry of the variety $\mathcal{U}_{n,H}^{\xi}$ defined in \ref{eq:2.9} for any $\boldsymbol{Z}\leq H\leq\boldsymbol{\mu}_{\boldsymbol{2}}^{n}$.
\begin{not*}
\emph{If $Z=\mathrm{diag}{\left(\varepsilon_{1},\ldots,\varepsilon_{n},\varepsilon_{1},\ldots,\varepsilon_{n}\right)}\in\boldsymbol{\mu}_{\boldsymbol{2}}^{n}$, we denote $Z$ by $\mathrm{diag^{2}}{\left(\varepsilon_{1},\ldots,\varepsilon_{n}\right)}$. If $\Phi\in S_{n}$, we call $\Phi$
the corresponding symplectic permutation matrix too. If $A\in\mathrm{Sp}{\left(2n,\mathbb{K}\right)}$,
we write $\Phi{\left(A\right)}$ instead of $\Phi A \Phi^{-1}$ as in Proposition \ref{prop:2.3}.}
\end{not*}
Let $H$ be a subgroup of $\boldsymbol{\mu}_{\boldsymbol{2}}^{n}/\boldsymbol{Z}$ of rank $k$, $\overrightarrow{\mathcal{B}}$ a basis of $H$. Then $\overrightarrow{\mathcal{B}}=\left\{Z_{1},\ldots,Z_{k}\right\} $, where $Z_{1},\ldots,Z_{k}$ are independent matrices, defined up to a sign, such that $-I_{2n}\notin\mathrm{span}\left\{\overline{Z_{1}},\ldots,\overline{Z_{k}}\right\}\leq\boldsymbol{\mu}_{\boldsymbol{2}}^{n}$, whatever the choice of representatives $\overline{Z_{1}},\ldots,\overline{Z_{k}}$ of  $Z_{1},\ldots,Z_{k}$ is.
It can be easily shown that there exists a permutation $\Phi\in S_{n}$
such that, for all $h\in[k]$,
\begin{equation}
\label{eq:2.15}
\Phi{\left(Z_{h}\right)}=\mathrm{diag^{2}}{\left(\overset{a_{1}^{h}}{\overbrace{1,\ldots,1}},\overset{a_{2}^{h}}{\overbrace{-1,\ldots,-1}},\ldots,\overset{a_{2^{h}-1}^{h}}{\overbrace{1,\ldots,1}},\overset{a_{2^{h}}^{h}}{\overbrace{-1,\ldots,-1}}\right)}
\end{equation}
where $a_{i}^{h-1}=a_{2i-1}^{h}+a_{2i}^{h}$ for any $i\in[2^{h-1}]$.

This permutation gives rise to the following family of unordered partitions of $n$:
\[
\left\{ \left(a_{1}^{h},\ldots,a_{2^{h}}^{h}\right)\right\} _{h\in[k]}.
\]
We will prove that these partitions are uniquely determined by the
subgroup $H$. In order to do this, we have to check the following
facts about the set $\left\{ \left(a_{1}^{h},\ldots,a_{2^{h}}^{h}\right)\right\} _{h\in[k]}$:
\begin{enumerate}
\item It does not depend on the choice of the permutation $\Phi$.\label{enu:2.1}
\item It does not depend on the choice of the representatives of the elements
of the basis $\overrightarrow{\mathcal{B}}$.\label{enu:2.2}
\item It does not depend on the choice of the basis $\overrightarrow{\mathcal{B}}$,
once the representatives of its elements are fixed.\label{enu:2.3}
\end{enumerate}
\begin{lem}
\label{lem:2.15}
Let $H\leq\boldsymbol{\mu}_{\boldsymbol{2}}^{n}$, $\overrightarrow{\mathcal{B}}=\left\{ Z_{1},\ldots,Z_{k}\right\} $
a basis of $H$ such that the elements of $\overrightarrow{\mathcal{B}}$
are of the form \ref{eq:2.15}. If $\lambda\in S_{n}$ such that $\lambda{\left(Z_{h}\right)}=Z_{h}$ for all $h\in[k]$,
then $\lambda\in\prod\limits_{\substack{i=1}}^{2^{k}}{S_{a_{i}^{k}}}$. 
\end{lem}
\begin{proof}
By induction on $k=\mathrm{rk}{\left(H\right)}$.

$k=1$: In this case, $\lambda{\left(Z_{1}\right)}=Z_{1}$,
and since $Z_{1}$ is of the form \ref{eq:2.15}, $\lambda\in S_{a_{1}^{1}}\times S_{a_{2}^{1}}$.

$k\mapsto k+1$: Let $\lambda\in S_{n}$ such that $\lambda{\left(Z_{h}\right)}=Z_{h}$ for any $h=1,\ldots,k+1$.
In particular $\lambda{\left(Z_{h}\right)}=Z_{h}$ for any $h\in[k]$.
By the inductive hypothesis, 
\[
\lambda=\left(\lambda_{1},\ldots,\lambda_{2^{k}}\right)\in\prod\limits_{\substack{i=1}}^{2^k}{S_{a_{i}^{k}}}.
\]
Now, $Z_{k+1}=\mathrm{diag^{2}}{\left(W_{1},\ldots,W_{2^{k}}\right)}$,
where 
\[
W_{i}=\mathrm{diag}{\left(\overset{a_{2i-1}^{k+1}}{\overbrace{1,\ldots,1}},\overset{a_{2i}^{k+1}}{\overbrace{-1,\ldots,-1}}\right)}
\]
for $i\in[2^{k}]$ and $\lambda{\left(Z_{k+1}\right)}=Z_{k+1}$. This implies that $\lambda_{i}{\left(W_{i}\right)}=W_{i}$ hence, by the case $k=1$, $\lambda_{i}\in S_{a_{2i-1}^{k+1}}\times S_{a_{2i}^{k+1}}$,
and this proves the lemma.
\end{proof}
\begin{not*}
\emph{If $\lambda\in S_{n}$ and $\overrightarrow{\mathcal{B}}=\left\{ Z_{1},\ldots,Z_{k}\right\} $
an ordered basis of a subspace $H$ of $\boldsymbol{\mu}_{\boldsymbol{2}}^{n}$, then $\lambda{\left(\overrightarrow{\mathcal{B}}\right)}:=\left\{ \lambda{\left(Z_{1}\right)},\ldots,\lambda{\left(Z_{k}\right)}\right\}$.}
\end{not*}
\begin{lem}
\label{lem:2.16}
Let $H\leq\boldsymbol{\mu}_{\boldsymbol{2}}^{n}$, $\overrightarrow{\mathcal{B}}=\left\{ Z_{1},\ldots,Z_{k}\right\}$
an ordered basis of $H$ and $\lambda,\,\mu\in S_{n}$ such that $\lambda{\left(\overrightarrow{\mathcal{B}}\right)}$
and $\mu{\left(\overrightarrow{\mathcal{B}}\right)}$ are of the form
\ref{eq:2.15}. Then $\lambda{\left(\overrightarrow{\mathcal{B}}\right)}=\mu{\left(\overrightarrow{\mathcal{B}}\right)}$.
\end{lem}
\begin{proof}
By induction on $k$.

$k=1$: The assertion is true because $a_{1}^{1}$ and $a_{2}^{1}$
are equal to the number of eigenvalues equal to $1$ and $-1$ in $Z_{1}$
respectively.

$k\mapsto k+1$: If $\overrightarrow{\mathcal{B}}=\left\{ Z_{1},\ldots,Z_{k+1}\right\}$
and $\lambda{\left(\overrightarrow{\mathcal{B}}\right)}$ and $\mu{\left(\overrightarrow{\mathcal{B}}\right)}$ are of the form \ref{eq:2.15}, then, by the inductive hypothesis,
$\lambda{\left(Z_{i}\right)}=\mu{\left(Z_{i}\right)}$ for any $i\in[k]$.
In other words, if 
\[
\begin{aligned}
&\left\{\left(a_{1}^{h},\ldots,a_{2^{h}}^{h}\right)\right\} _{h\in[k+1]},\\
&\left\{\left(b_{1}^{h},\ldots,b_{2^{h}}^{h}\right)\right\} _{h\in[k+1]}
\end{aligned}
\]
are the partitions of $n$ determined by $\lambda$ and $\mu$ respectively, then
$a_{i}^{h}=b_{i}^{h}$ for $h\in[k]$.

\noindent
Now, ${\left(\lambda\mu^{-1}\right)}{\left(\mu{\left(Z_{i}\right)}\right)}=\lambda{\left(Z_{i}\right)}=\mu{\left(Z_{i}\right)}$ for any $i\in[k]$.
By Lemma \ref{lem:2.15}, $\lambda\mu^{-1}=\left(\alpha_{1},\ldots,\alpha_{2^{k}}\right)\in\prod\limits_{\substack{i=1}}^{2^k}{S_{a_{i}^{k}}}$. Write $\lambda{\left(Z_{k+1}\right)}=\mathrm{diag^{2}}{\left(W_{1},\ldots,W_{2^{k}}\right)}$ and $\mu{\left(Z_{k+1}\right)}=\mathrm{diag^{2}}{\left(V_{1},\ldots,V_{2^{k}}\right)}$,
where
\[
\begin{aligned}
W_{i}&=\mathrm{diag}{\left(\overset{a_{2i-1}^{k+1}}{\overbrace{1,\ldots,1}},\overset{a_{2i}^{k+1}}{\overbrace{-1,\ldots,-1}}\right)},\\
V_{i}&=\mathrm{diag}{\left(\overset{b_{2i-1}^{k+1}}{\overbrace{1,\ldots,1}},\overset{b_{2i}^{k+1}}{\overbrace{-1,\ldots,-1}}\right)}
\end{aligned}
\]
for $i\in[2^{k}]$. We have that 
\[
\begin{aligned}
\lambda{\left(Z_{k+1}\right)}&={\left(\lambda\mu^{-1}\right)}{\left(\mu{\left(Z_{k+1}\right)}\right)}\\
&=\mathrm{diag^{2}}{\left(\alpha_{1}{\left(V_{1}\right)},\ldots,\alpha_{2^{k}}{\left(V_{2^{k}}\right)}\right)}=\mathrm{diag^{2}}{\left(W_{1},\ldots,W_{2^{k}}\right)}.
\end{aligned}
\]
It follows that the $\alpha_{i}{\left(V_{i}\right)}$'s have
the form \ref{eq:2.15}, and by the case $k=1$, $\alpha_{i}{\left(V_{i}\right)}=V_{i}$
for $i\in[2^{k}]$. So $V_{i}=W_{i}$ for any $i\in[2^{k}]$
and this concludes the proof.
\end{proof}
\begin{lem}
\label{lem:2.17}
Let $H\leq\boldsymbol{\mu}_{\boldsymbol{2}}^{n}$ such that $-I_{2n}\notin H$,
$\overrightarrow{\mathcal{B}}=\left\{ Z_{1},\ldots,Z_{k}\right\} $
an ordered basis of $H$, $\overrightarrow{\mathcal{B^{\prime}}}=\left\{ Z_{1},\ldots,Z_{j-1},-Z_{j},Z_{j+1},\ldots,Z_{k}\right\} $
and 
\begin{equation}
\label{eq:parti}
\begin{aligned}
&\left\{ \left(a_{1}^{h},\ldots,a_{2^{h}}^{h}\right)\right\} _{h\in[k]},\\
&\left\{ \left(b_{1}^{h},\ldots,b_{2^{h}}^{h}\right)\right\} _{h\in[k]}
\end{aligned}
\end{equation}
the partitions of $n$ associated to $\overrightarrow{\mathcal{B}}$
and $\overrightarrow{\mathcal{B^{\prime}}}$ respectively. Then for any $h\in[k]$,
there exists a permutation $\lambda_{h}\in S_{2^{h}}$ such that $b_{i}^{h}=a_{\lambda_{h}{\left(i\right)}}^{h}$
for all $i\in[2^{h}]$.
\end{lem}
\begin{proof}
First of all, notice that partitions \ref{eq:parti} are uniquely determined by Lemma \ref{lem:2.16}. Let $\mu\in S_{n}$ such that $\mu{\left(\overrightarrow{\mathcal{B}}\right)}$ is of the
form \ref{eq:2.15}. Then the matrices of $\mu{\left(\overrightarrow{\mathcal{B}^{\prime}}\right)}$
assume the following form:
\[
\mu{\left(Z_{h}\right)}=\mathrm{diag^{2}}{\left(\overset{a_{1}^{h}}{\overbrace{1,\ldots,1}},\overset{a_{2}^{h}}{\overbrace{-1,\ldots,-1}},\ldots,\overset{a_{2^{h}-1}^{h}}{\overbrace{1,\ldots,1}},\overset{a_{2^{h}}^{h}}{\overbrace{-1,\ldots,-1}}\right)}
\]
for $h\neq j$ and 
\[
\mu{\left(Z_{j}\right)}=\mathrm{diag^{2}}{\left(\overset{a_{1}^{j}}{\overbrace{-1,\ldots,-1}},\overset{a_{2}^{j}}{\overbrace{1,\ldots,1}},\ldots,\overset{a_{2^{j}-1}^{j}}{-\overbrace{1,\ldots,-1}},\overset{a_{2^{j}}^{j}}{\overbrace{1,\ldots,1}}\right)}.
\]
Let $\lambda\in S_{n}$ such that, for any $l\in[2^{j}]$,
if $\left(\sum\limits_{\substack{i=1}}^{l-1}a_{i}^{j}+1\right)\leq s\leq\sum\limits_{\substack{i=1}}^{l}a_{i}^{j}$,
then 
\[
\lambda{\left(s\right)}=\begin{cases}
s+a_{l+1}^{j} & \text{if $l$ is odd}\\
s-a_{l-1}^{j} & \text{if $l$ is even}
\end{cases}
\]
($\lambda$ is the permutation that exchanges pairwise the
blocks of size $a_{2i-1}^{j}$, $a_{2i}^{j}$, $i\in[2^{j-1}]$).
It is easy to check that $\lambda\mu{\left(\overrightarrow{\mathcal{B}^{\prime}}\right)}$
is of the form \ref{eq:2.15} and that the desired permutations $\lambda_{h}$
are the following:
\[
\lambda_{h}=\begin{cases}
\mathrm{id}_{\left[n\right]} & \text{if $1\leq h\leq j-1$}\\
\prod\limits_{\substack{l=0}}^{2\left(2^{j}-1\right)}{\left(\prod\limits_{\substack{i=1}}^{2^{h-j}}{\left(i+l2^{h-j},i+\left(l+1\right)2^{h-j}\right)}\right)} & \text{if $j\leq h\leq k$.}
\end{cases}
\]
\end{proof}

\begin{lem}
\label{lem:2.18}
Let $H$ and $\overrightarrow{\mathcal{B}}$
like in Lemma \ref{lem:2.17}, $\overrightarrow{\mathcal{B}^{\prime}}$
another ordered basis of $H$ and consider the partitions of $n$ as in \ref{eq:parti} associated to $\overrightarrow{\mathcal{B}}$ and $\overrightarrow{\mathcal{B^{\prime}}}$.Then for any $h\in[k]$, there exists a permutation $\lambda_{h}\in S_{2^{h}}$ such that $b_{i}^{h}=a_{\lambda_{h}{\left(i\right)}}^{h}$
for all $i\in[2^{h}]$.
\end{lem}
\begin{proof}
First of all, let us give an explicit description of the
action of an element in $\mathrm{GL}{\left(k,\mathbb{Z}_{2}\right)}$
on an ordered basis of $H$: if $A=\left(a_{ij}\right)_{i,j\in[k]}\in\mathrm{GL}{\left(k,\mathbb{Z}_{2}\right)}$ and $\overrightarrow{\mathcal{B}}=\left\{ Z_{1},\ldots,Z_{k}\right\}$,
then
\[ 
A{\left(\overrightarrow{\mathcal{B}}\right)}:=\left\{\prod\limits_{\substack{i=1}}^kZ_{i}^{a_{ij}}\right\} _{j\in[k]}.
\]
It is known that $\mathrm{GL}{\left(k,\mathbb{Z}_{2}\right)}$ is generated
as a group by the matrices of the form $I_{k}+E_{i,i+1}$ for $i=1,\ldots,k-1$,
where $E_{i,i+1}$ is the $\left(i,i+1\right)$-th elementary matrix.
Then, it is sufficient to prove the lemma only for the change of basis
given by these matrices. For simplicity, we will give the proof only
in the case where the change of basis is given by $I_{k}+E_{1,2}$,
being the proof in the other cases completely analogous. Therefore,
the basis involved are 
\[
\begin{aligned}
\overrightarrow{\mathcal{B}}&=\left\{ Z_{1},\ldots,Z_{k}\right\},\\
\overrightarrow{\mathcal{B^{\prime}}}&=\left\{ Z_{1},Z_{1}Z_{2},Z_{3},\ldots,Z_{k}\right\}.
\end{aligned}
\]
Let $\mu\in S_{n}$ such that $\mu{\left(\overrightarrow{\mathcal{B}}\right)}$
is of the form \ref{eq:2.15}. Then
\[
\mu{\left(Z_{1}Z_{2}\right)}=\mathrm{diag^{2}}{\left(\overset{a_{1}^{2}}{\overbrace{1,\ldots,1}},\overset{a_{2}^{2}}{\overbrace{-1,\ldots,-1}},\overset{a_{3}^{2}}{\overbrace{-1,\ldots,-1}},\overset{a_{4}^{2}}{\overbrace{1,\ldots,1}}\right)}.
\]
Let $\lambda\in S_{n}$ such that 
\[
\lambda{\left(s\right)}=\begin{cases}
s & \text{if $1\leq  s\leq a_{1}^{2}+a_{2}^{2}$}\\
s+a_{4}^{2} & \text{if $a_{1}^{2}+a_{2}^{2}+1\leq s\leq a_{1}^{2}+a_{2}^{2}+a_{3}^{2}$}\\
s-a_{3}^{2} & \text{otherwise}
\end{cases}
\]
($\lambda$ is the permutation exchanging the blocks of
size $a_{3}^{2}$, $a_{4}^{2}$). Then $\lambda\mu{\left(\overrightarrow{\mathcal{B}^{\prime}}\right)}$
is of the form \ref{eq:2.15} and it turns out that $a_{1}^{1}=b_{1}^{1}$,
$a_{2}^{1}=b_{2}^{1}$ and that, for $h=2,\ldots,k$,
\[
b_{i}^{h}=\begin{cases}
a_{i}^{h} & \text{if $1\leq i\leq2^{h-1}$}\\
a_{i+2^{h-2}}^{h} & \text{if $2^{h-1}+1\leq i\leq2^{h-1}+2^{h-2}$}\\
a_{i-2^{h-2}}^{h} & \text{if $2^{h-1}+2^{h-2}+1\leq i\leq2^{h}$.}
\end{cases}
\]
\end{proof}
Summaring up, Lemma \ref{lem:2.15} together with Lemma \ref{lem:2.16} prove \ref{enu:2.1}, Lemma \ref{lem:2.17} proves \ref{enu:2.2} and Lemma \ref{lem:2.18} proves \ref{enu:2.3}.

Finally, we are able to prove the following
\begin{thm}
\label{thm:2.19}
Let $H$ be a subgroup of $\boldsymbol{\mu}_{\boldsymbol{2}}^{n}$
containing $\boldsymbol{Z}$. There exists a unique set-partition $\left\{ \Pi_{i}\right\} _{i\in[2^{k}]}$
of $\Pi=\left\{ \varphi^{m_{1}},\ldots,\varphi^{m_{n}}\right\} $
such that 
\begin{equation}
\label{eq:theo}
\mathcal{U}_{n,H}^{\xi}\cong\prod\limits_{\substack{i=1}}^{2^{k}}{\mathcal{U}_{a_{i}^{k}}^{\Pi_{i}{\left(\xi\right)}}}
\end{equation}
where 
\[
\Pi_{i}{\left(\xi\right)}:=\mathrm{diag}{\left(\left(\xi_{h}\right)_{h\in\Pi_{i}},\left(\xi_{h}^{-1}\right)_{h\in\Pi_{i}}\right)}
\]
and $a_{i}^{k}=\left|\Pi_{i}\right|$ for any $i\in[2^{k}]$.
\end{thm}
\begin{rem}
\label{rem:2.20}
For any $i\in[2^{k}]$, $\Pi_{i}{\left(\xi\right)}$ is a matrix
uniquely determined by $\Pi_{i}$ up the action of a permutation $\varphi\in S_{a_{i}^{k}}$.
So $\mathcal{U}_{a_{i}^{k}}^{\Pi_{i}{\left(\xi\right)}}$ is uniquely
determined up to isomorphism induced by such a $\varphi$. 
\end{rem}

\begin{proof}[Proof of Theorem \ref{thm:2.19}]
Notice that for $i\in[2^{k}]$, the eigenvalues of
$\Pi_{i}{\left(\xi\right)}$ satisfy \ref{eq:2.1} because of Remark \ref{rem:sub}, so $\mathcal{U}_{a_{i}^{k}}^{\Pi_{i}\left(\xi\right)}$
is well defined. Let $X=\left(A_{1},B_{1},\ldots,A_{g},B_{g}\right)\in\mathcal{U}_{n}^{\xi}$,
$\overrightarrow{\mathcal{B}}=\left\{ Z_{1},\ldots,Z_{k}\right\} $
an ordered basis of $H/\boldsymbol{Z}$. Then $X\in\mathcal{U}_{n,H}^{\xi}$
if and only if, for $i\in[g]$ and $j\in[k]$,  
\begin{equation}
\label{eq:equazioni}
\begin{aligned}
A_{i}Z_{j}&=Z_{j}A_{i},\\
B_{i}Z_{j}&=Z_{j}B_{i}.
\end{aligned}
\end{equation}
Applying a permutation $\Phi$
such that $\Phi{\left(\overrightarrow{\mathcal{B}}\right)}$ is of the
form \ref{eq:2.15}  to equations \ref{eq:equazioni}, we have, for $i\in[g]$ and $j\in[k]$,
\[
\begin{aligned}
\Phi{\left(A_{i}\right)}\Phi{\left(Z_{j}\right)}&=\Phi{\left(Z_{j}\right)}\Phi{\left(A_{i}\right)},\\
\Phi{\left(B_{i}\right)}\Phi{\left(Z_{j}\right)}&=\Phi{\left(Z_{j}\right)}\Phi{\left(B_{i}\right)}
\end{aligned}
\]
For any $i\in[g]$, it can be easily shown that
\[
\Phi{\left(A_{i}\right)}=\begin{pmatrix}
A_{i}^{1} & A_{i}^{2}\\
A_{i}^{3} & A_{i}^{4}
\end{pmatrix},\,\,\Phi{\left(B_{i}\right)}=\begin{pmatrix}
B_{i}^{1} & B_{i}^{2}\\
B_{i}^{3} & B_{i}^{4}
\end{pmatrix}
\]
where 
\[
\begin{aligned}
A_{i}^{s}&=\mathrm{diag}{\left(C_{i,1}^{s},\ldots,C_{i,2^{k}}^{s}\right)},\\
B_{i}^{j}&=\mathrm{diag}{\left(D_{i,1}^{s},\ldots,D_{i,2^{k}}^{s}\right)}
\end{aligned}
\]
and $C_{i,h}^{s}$, $D_{i,h}^{s}$ are square matrices of size $a_{h}^{k}$
for any $s=1,\ldots,4$, $h\in[2^{k}]$.
Since the $\Phi{\left(A_{i}\right)}$'s and the $\Phi{\left(B_{i}\right)}$'s are
symplectic matrices, we have for $i\in[g]$, $h\in[2^{k}]$
\begin{equation}
\label{eq:2.20}
\begin{aligned}
&\left(C_{i,h}^{1}\right)^{T}{C_{i,h}^{4}}-\left(C_{i,h}^{3}\right)^{T}C_{i,h}^{2}=I_{a_{h}^{k}},\\
&\left(C_{i,h}^{3}\right)^{T}C_{i,h}^{1}=\left(C_{i,h}^{1}\right)^{T}C_{i,h}^{3},\\
&\left(D_{i,h}^{1}\right)^{T}D_{i,h}^{4}-\left(D_{i,h}^{3}\right)^{T}D_{i,h}^{2}=I_{a_{h}^{k}},\\
&\left(D_{i,h}^{3}\right)^{T}D_{i,h}^{1}=\left(D_{i,h}^{1}\right)^{T}D_{i,h}^{3}.
\end{aligned}
\end{equation}
Moreover, $\Phi$ determines a partition of $\left\{ \Pi_{i}\right\} _{i\in[2^{k}]}$
of $\Pi$, with $a_{i}^{k}=\left|\Pi_{i}\right|$, and 
\[
\Phi{\left(\xi\right)}=\mathrm{diag}{\left(\left(\xi_{h}\right)_{h\in\Pi_{1}},\ldots,\left(\xi_{h}\right)_{h\in\Pi_{2^{k}}},\left(\xi_{h}^{-1}\right)_{h\in\Pi_{1}},\ldots,\left(\xi_{h}^{-1}\right)_{h\in\Pi_{2^{k}}}\right)}.
\]
Now, let $\lambda\in S_{2n}$ such that 
\[
{\left(\lambda\Phi\right)}{\left(Z_{j}\right)}=\mathrm{diag}{\left(\overset{2a_{1}^{j}}{\overbrace{1,\ldots,1}},\overset{2a_{2}^{j}}{\overbrace{-1,\ldots,-1}},\ldots,\overset{2a_{2^{j}-1}^{j}}{\overbrace{1,\ldots,1}},\overset{2a_{2^{j}}^{j}}{\overbrace{-1,\ldots,-1}}\right)}
\]
for all $j\in[k]$, and $\lambda$ does not move any
element in any of the blocks of size $a_{h}^{j}$. Then 
\[
\begin{aligned}
{\left(\lambda\Phi\right)}{\left(A_{i}\right)}&=\mathrm{diag}{\left(C_{i}^{1},\ldots,C_{i}^{2^{k}}\right)},\\
{\left(\lambda\Phi\right)}{\left(B_{i}\right)}&=\mathrm{diag}{\left(D_{i}^{1},\ldots,D_{i}^{2^{k}}\right)}
\end{aligned}
\]
where 
\[
C_{i}^{h}=\begin{pmatrix}
C_{i,h}^{1} & C_{i,h}^{2}\\
C_{i,h}^{3} & C_{i,h}^{4}
\end{pmatrix},\,\,D_{i}^{h}=\begin{pmatrix}
D_{i,h}^{1} & D_{i,h}^{2}\\
D_{i,h}^{3} & D_{i,h}^{4}
\end{pmatrix}
\]
for $h\in[2^{k}]$, $i\in[g]$, so the $C_{i}^{h}$'s
and the $D_{i}^{h}$'s are symplectic by \ref{eq:2.20}, and 
\[
{\left(\lambda\Phi\right)}{\left(\xi\right)}=\mathrm{diag}{\left(\Pi_{1}{\left(\xi\right)},\ldots,\Pi_{2^{k}}{\left(\xi\right)}\right)}.
\]
It follows that there is an isomorphism between $\mathcal{U}_{n,H}^{\xi}$
and $\prod\limits_{\substack{i=1}}^{2^{k}}\mathcal{U}_{a_{i}^{k}}^{\Pi_{i}{\left(\xi\right)}}$
given by 
\begin{equation}
\begin{aligned}
f:\mathcal{U}_{n,H}^{\xi}&\overset{\cong}{\longrightarrow}\prod\limits_{\substack{i=1}}^{2^{k}}{\mathcal{U}_{a_{i}^{k}}^{\Pi_{i}\left(\xi\right)}}\\
\left(A_{1},B_{1},\ldots,A_{g},B_{g}\right)&\longmapsto\left(\left(C_{i}^{1},D_{i}^{1}\right)_{i=1,\ldots g},\ldots,\left(C_{i}^{2^{k}},D_{i}^{2^{k}}\right)_{i=1,\ldots,g}\right)
\end{aligned}
\end{equation}
induced by the permutation $\lambda\Phi$. 

If we choose a different $\Phi^{\prime}$ such that $\Phi^{\prime}{\left(\overrightarrow{\mathcal{B}}\right)}$
is of the form \ref{eq:2.15}, then by Lemma \ref{lem:2.15}, $\Phi^{-1}\Phi^{\prime}=\left(\alpha_{1},\ldots,\alpha_{2^{k}}\right)\in\prod\limits_{\substack{i=1}}^{2^{k}}{S_{a_{i}^{k}}}$.
Therefore, $\Phi^{\prime}$ induces the same partition $\left\{ \Pi_{i}\right\} _{i\in[2^{k}]}$
of $\Pi$ and by \ref{rem:2.20} we are done. If $\overrightarrow{\mathcal{B}^{\prime}}$
is a different basis of $H/\boldsymbol{Z}$, by the proofs of Lemma \ref{lem:2.17}
and \ref{lem:2.18}, we have that if $\Phi^{\prime}{\left(\overrightarrow{\mathcal{B}^{\prime}}\right)}$
is of the form \ref{eq:2.15}, $\Phi^{\prime}\in S_{n}$, then $\Phi^{\prime}=\mu\Phi$,
where $\mu\in S_{n}$ permutes the blocks of size $a_{i}^{h}$. It
follows that $\Phi^{\prime}$ induces the same partition $\left\{ \Pi_{i}\right\} _{i\in[2^{k}]}$
of $\Pi$ and we are done again.
\end{proof}

\section{$E$-polynomial of $\mathcal{M}_{n}^{\xi}$}
\label{section:4}

In the following section, we compute the $E$-polynomial of $\mathcal{M}_{n}^{\xi}/\mathbb{C}$
proving that, for any $\boldsymbol{Z}\leq H\leq\boldsymbol{\mu}_{\boldsymbol{2}}^{n}$, $\widetilde{\mathcal{M}}_{n,H}^{\xi}$
has polynomial counting functions over finite fields possessing a
primitive $2m$-th root of unity. These are the $E$-polynomials thanks
to Theorem \ref{thm:7}. Finally, we use the additiveness of the $E$-polynomial
with respect to stratifications.

\subsection{The $\mathbb{F}_{q}$-points of $\widetilde{\mathcal{U}}_{n,H}^{\xi}$}

\noindent Let $q$ be a power of a prime $p>3$ . Assume that
$\mathbb{F}_{q}$ contains a primitive $2m$-th root of unity $\zeta:=\varphi^{\frac{1}{2}}$,
so $q\underset{2m}{\equiv}1$. For any $\boldsymbol{Z}\leq H\leq\boldsymbol{\mu}_{\boldsymbol{2}}^{n}$,
define 
\begin{equation}
\widetilde{N}_{n,H}^{\xi}{\left(q\right)}:=\left|\widetilde{\mathcal{U}}_{n,H}^{\xi}{\left(\mathbb{F}_{q}\right)}\right|\label{eq:3.1.1}
\end{equation}
\begin{equation}
\label{eq:3.1.2}
\,\,\,N_{n,H}^{\xi}{\left(q\right)}:=\left|\mathcal{U}_{n,H}^{\xi}{\left(\mathbb{F}_{q}\right)}\right|.
\end{equation}

These quantities are the number of rational points of $\widetilde{\mathcal{U}}_{n,H}^{\xi}/\overline{\mathbb{F}_{q}}$
and $\mathcal{U}_{n,H}^{\xi}/\overline{\mathbb{F}_{q}}$ respectively.
When $H=\boldsymbol{Z}$, we simply write $\widetilde{N}_{n}^{\xi}{\left(q\right)}$
and $N_{n}^{\xi}{\left(q\right)}$. From Definition \ref{def:2.4},
it is easy to see that 
\begin{equation}
N_{n,H}^{\xi}{\left(q\right)}=\sum\limits_{\substack{H\leq S\leq\boldsymbol{\mu}_{\boldsymbol{2}}^{n}}}{\widetilde{N}_{n,S}^{\xi}{\left(q\right)}}\label{eq:3.1.3}
\end{equation}
hence by Möbius inversion, we have
\begin{equation}
\widetilde{N}_{n,H}^{\xi}{\left(q\right)}=\sum\limits_{\substack{H\leq S\leq\boldsymbol{\mu}_{\boldsymbol{2}}^{n}}}{\mu{\left(H,S\right)}N_{n,S}^{\xi}{\left(q\right)}}\label{eq:3.1.4}
\end{equation}
where $\mu$ is the Möbius function of the poset of the
subgroups of $\boldsymbol{\mu}_{\boldsymbol{2}}^{n}$ and it is defined as follows
\begin{equation}
\mu{\left(H,S\right)}=
\begin{cases}
\left(-1\right)^{\mathrm{rk}{\left(S\right)}-\mathrm{rk}{\left(H\right)}}2^{\tbinom{\mathrm{rk}{\left(S\right)}-\mathrm{rk}{\left(H\right)}}{2}} & \text{if $H\subseteq S$}\\
0 & \text{otherwise.}
\end{cases}
\label{eq:3.15}
\end{equation}
By Theorem \ref{thm:2.19}, if $\boldsymbol{Z}\leq S\leq\boldsymbol{\mu}_{\boldsymbol{2}}^{n}$,
there exists a unique partition $\left\{ \Pi_{i}\right\} _{i\in[2^{\mathrm{rk}{\left(S\right)}}]}$
of $\Pi=\left\{ \varphi^{m_{1}},\ldots,\varphi^{m_{n}}\right\} $
such that 
\begin{equation}
N_{n,S}^{\xi}{\left(q\right)}=\prod_{i=1}^{2^{\mathrm{rk}{\left(S\right)}}}{N_{a_{i}^{\mathrm{rk}{\left(S\right)}}}^{\Pi_{i}{\left(\xi\right)}}{\left(q\right)}}\label{eq:3.1.6}
\end{equation}
where $a_{i}^{\mathrm{rk}{\left(S\right)}}=\left|\Pi_{i}\right|$ for any
$i=1,\ldots,2^{\mathrm{rk}\left(S\right)}$. Plugging \ref{eq:3.15} and \ref{eq:3.1.6} into
\ref{eq:3.1.4}, we obtain
\begin{equation}
\widetilde{N}_{n,H}^{\xi}{\left(q\right)}=\sum\limits_{\substack{H\leq S\leq\boldsymbol{\mu}_{\boldsymbol{2}}^{n}}}{\left(-1\right)^{\mathrm{rk}{\left(S\right)}-\mathrm{rk}{\left(H\right)}}2^{\tbinom{\mathrm{rk}{\left(S\right)}-\mathrm{rk}{\left(H\right)}}{2}}\prod_{i=1}^{2^{\mathrm{rk}{\left(S\right)}}}{N_{a_{i}^{\mathrm{rk}{\left(S\right)}}}^{\Pi_{i}{\left(\xi\right)}}{\left(q\right)}}}.\label{eq:3.1.7}
\end{equation}
The advantage of this equation is that we can compute $N_{a_{i}^{\mathrm{rk}{\left(S\right)}}}^{\Pi_{i}{\left(\xi\right)}}{\left(q\right)}$
using the following formula due to Frobenius (see \cite{key-17}, \cite{key-18} or \cite{key-19}).
\begin{prop}
\label{prop:frobenius}
Let $G$ be a finite group. Given
$z\in G$, the number of $2g$-tuples $\left(x_{1},y_{1},\ldots,x_{g},y_{g}\right)$
satisfying $\prod_{i=1}^{g}{\left[x_{i}:y_{i}\right]}{z}=1$
is:
\begin{equation}
\left|\left\{ \prod_{i=1}^{g}{\left[x_{i}:y_{i}\right]}{z}=1\right\} \right|=\sum\limits_{\substack{\chi\in\mathrm{Irr}{\left(G\right)}}}{\chi{\left(z\right)}{\left(\frac{\left|G\right|}{\chi{\left(1\right)}}\right)}^{2g-1}}.\label{eq:3.1.8}
\end{equation}
\end{prop}
However, our next goal is to compute $N_{n}^{\xi}{\left(q\right)}$.
We specialize \ref{eq:3.1.8} to the case where $G=\mathrm{Sp}{\left(2n,\mathbb{F}_{q}\right)}$
and $z=\xi$. 

Since $\xi$ is a regular semisimple matrix, the range of the summation
in \ref{eq:3.1.8} restricts to the set of the principal series of
$\mathrm{Sp}{\left(2n,\mathbb{F}_{q}\right)}$ by \ref{eq:1.4}.
By Proposition \ref{prop:25}, we can collect principal series of the same degree according to the type $\tau$, so combining
with \ref{eq:1.12}, we obtain that
\[
N_{n}^{\xi}{\left(q\right)}=\sum\limits_{\substack{\tau}}{\left(H_{\tau}{\left(q\right)}\right)^{2g-1}C_{\tau}}
\]
where
\[
C_{\tau}:=\sum\limits_{\substack{\tau{\left(\chi\right)}=\tau}}{\chi{\left(\xi\right)}}.
\]
Our next task is to compute $C_{\tau}$; we will find that it is an
integer constant. In particular, since the number of all possibile
types does not depend on $q$, this will show that $N_{n}^{\xi}{\left(q\right)}\in\mathbb{Z}[q]$.
\begin{rem}
For even values of $m$ and $q\underset{m}{\equiv}1$, we could not
get $N_{n}^{\xi}{\left(q\right)}$ to be a polynomial in $q$. In fact,
let $\xi=\begin{pmatrix}
i & 0\\
0 & -i
\end{pmatrix}$, $m=4$ and let us compute $N_{1}^{\xi}{\left(q\right)}$ when $q\underset{4}{\equiv}1$.
Using the well known character table of $\mathrm{SL}{\left(2,\mathbb{F}_{q}\right)}$ that one can find in \cite{key-12},
after some little algebra we get the quasi polynomial:
\begin{equation}
\label{eq:1.1.4}
\begin{aligned}
N_{1}^{\xi}{\left(q\right)}=&\left(q^{3}-q\right)^{2g-1}+\left(q^{2}-1\right)^{2g-1}\\
&+\left(\left(-1\right)^{\frac{q-1}{4}}\left(2^{2g}-1\right)-1\right)\left(q^{2}-q\right)^{2g-1}.
\end{aligned}
\end{equation}
This motivates our requirement on $q$ to be equal to $1$ modulo
$2m$.
\end{rem}

\subsection{\label{subsec:C tau}Calculation of $C_{\tau}$}

We refer to the notations used in subsection \ref{subsec:1}. If $\tau=\left(\lambda,\alpha_{1},\alpha_{\epsilon},\beta\right)$,
with $c:=\left|\lambda\right|$, $l:=l{\left(\lambda\right)}$, and
$\beta\in\mathrm{Irr}{\left(S_{\lambda,\alpha_{1},\alpha_{\epsilon}}\right)}$,
then combining Remark \ref{rem:13} and Proposition \ref{prop:19}
in formula \ref{eq:1.4} we have 
\[
\chi{\left(\xi\right)}=\frac{1}{\left|S_{\lambda,\alpha_{1},\alpha_{\epsilon}}\right|}\sum\limits_{\substack{w\in W_{n}}}{\theta^{w}{\left(\xi^{-1}\right)}\beta{\left(1\right)}}
\]
where $\theta\in\widehat{T^{F}}$ is of the form \ref{eq:1.7}. Define,
for $i=1,\ldots,n$ and $k\in Q$,
\[
\gamma_{km_{i}}:=\varphi^{km_{i}}+\varphi^{-km_{i}}.
\]
where we denote the complex counterpart of $\gamma$ in the same way. Then, after some computations, we obtain the following expression
for $C_{\tau}$:
\begin{equation}
C_{\tau}=\sum\limits_{\substack{\pi\in\mathcal{P}_{\lambda}}}{\sum\limits_{\substack{k_{1},\ldots,k_{l}\in Q \\ k_{s}\neq k_{t} \\ s\neq t}}{\left(\prod\limits_{\substack{j=1}}^{l}{\prod\limits_{\substack{s\in I_{j}}}{\gamma_{k_{j}m_{s}}}}\right)}}
\end{equation}
where $\mathcal{P}_{\lambda}$ is the set of partitions $\pi=\coprod\limits_{\substack{j=1}}^l{I_{j}}$
of $[c]$ such that $\left|I_{j}\right|=\lambda_{j}$ for any $j=1,\ldots,l$.

If $\sigma=\coprod\limits_{\substack{j=1}}^{l\left(\sigma\right)}I_{j}^{\prime}\in\Pi_{c}$,
let us consider the sets $\Sigma_{\sigma}$ and $\Sigma_{\sigma}^{\prime}$
as in \ref{eq:sigma}, replacing $\left[x\right]$ with $Q$, and define 
\begin{equation}
\Psi{\left(\sigma\right)}:=\sum\limits_{\substack{h\in\Sigma_{\sigma}}}{\left(\prod\limits_{\substack{j=1}}^{l{\left(\sigma\right)}}{\prod\limits_{\substack{s\in I_{j}^{\prime}}}{\gamma_{h{\left(I_{j}^{\prime}\right)}m_{s}}}}\right)}\label{eq:3.10}
\end{equation}
\begin{equation}
\label{eq:3.1.11}
\Phi{\left(\sigma\right)}:=\sum\limits_{\substack{h\in\Sigma_{\sigma}^{\prime}}}{\left(\prod\limits_{\substack{j=1}}^{l{\left(\sigma\right)}}{\prod\limits_{\substack{s\in I_{j}^{\prime}}}{\gamma_{h{\left(I_{j}^{\prime}\right)}m_{s}}}}\right)}
\end{equation}
It is evident that 
\begin{equation}
C_{\tau}=\sum\limits_{\substack{\pi\in\mathcal{P}_{\lambda}}}{\Phi{\left(\pi\right)}}\label{eq:3.1.12}
\end{equation}
and that $\Psi{\left(\pi\right)}=\sum\limits_{\substack{\pi\preceq\sigma}}{\Phi{\left(\sigma\right)}}$.
By Möbius inversion applied on the poset of set-partitions of $[c]$, we have
\begin{equation}
\Phi{\left(\pi\right)}=\sum\limits_{\substack{\pi\preceq\sigma}}{\mu{\left(\pi,\sigma\right)}\Psi{\left(\sigma\right)}}.
\label{eq:3.1.13}
\end{equation}
Interchanging sum and product in \ref{eq:3.10}, we have $\Psi{\left(\sigma\right)}=\prod\limits_{\substack{j=1}}^{l\left(\sigma\right)}{\Delta_{j}}$
with 
\begin{equation}
\label{eq:3.1.15}
\Delta_{j}:=\sum\limits_{\substack{k\in Q}}{\left(\prod\limits_{\substack{s\in I_{j}^\prime}}{\gamma_{km_{s}}}\right)}.
\end{equation}
Since $\varphi^{m_{1}},\ldots,\varphi^{m_{n}}$ satisfy \ref{eq:2.1}, together
with Remark \ref{rem:sub} we deduce that 
\begin{equation}
\label{eq:3.1.16}
\prod\limits_{\substack{s\in I_{j}^{\prime}}}{\gamma_{km_{s}}}=\sum\limits_{\substack{i=1}}^{2^{\lambda_{j}^{\prime}}}{\varphi_{i}^{k}}
\end{equation}
where $\lambda_{j}^{\prime}=\left|I_{j}^{\prime}\right|$ and $\varphi_{i}$'s
are primitive $k_{i}$-th roots of unity with $k_{i}>1$, $k_{i}\rvert m$.
Now, $q\underset{2m}{\equiv}1$ implies that $\left|Q\right|=\frac{q-3}{2}\underset{m}{\equiv}-1$,
so plugging \ref{eq:3.1.16} into \ref{eq:3.1.15}, we get $\Delta_{j}=-2^{\lambda_{j}^{\prime}}$ and
then $\Psi{\left(\sigma\right)}=\left(-1\right)^{l{\left(\sigma\right)}}2^{c}$.
Plugging it into \ref{eq:3.1.13} and using \ref{eq:2}, we obtain
that $\Phi{\left(\pi\right)}=2^{c}{\left(-1\right)}^{l}l!$ and consequently,
from \ref{eq:3.1.12}, 
\begin{equation}
C_{\tau}=\frac{n!\beta{\left(1\right)}2^{c}\left(-1\right)^{l}l!}{\prod\limits_{\substack{i}}{m_{i}{\left(\lambda\right)}!}\prod\limits_{\substack{i=1}}^{l}{\lambda_{i}!}\alpha_{1}!\alpha_{\varepsilon}!}=\frac{\left(-1\right)^{l}l!\left[W_{n}:S_{\lambda,\alpha_{1},\alpha_{\epsilon}}\right]\beta{\left(1\right)}}{\prod\limits_{\substack{i}}{m_{i}{\left(\lambda\right)}!}}\label{eq:3.1.14}
\end{equation}
that is an integer number. Here, $m_{i}(\lambda)=\left\{j\mid\lambda_{j}=i\right\}$, that is the multiplicity of $i$ in the partition $\lambda$, for any possible $i$.
\begin{rem}
\label{rem:.3.2}As one can see, the computation of the coefficients
$C_{\tau}$'s does not depend on the choice of $\varphi$ and $m_{1},\ldots,m_{n}$,
but only on the fact that $\varphi^{m_{1}},\ldots,\varphi^{m_{n}}$
satisfy \ref{eq:2.1}. Therefore, by Remark \ref{rem:2.1}, we get $N_{a_{i}^{\mathrm{rk}{\left(S\right)}}}^{\Pi_{i}{\left(\xi\right)}}{\left(q\right)}\in\mathbb{Z}[q]$
for all possible cases. 
\end{rem}

\subsection{Main formula}

For $\boldsymbol{Z}\leq H\leq\boldsymbol{\mu}_{\boldsymbol{2}}^{n}$, let
\begin{equation}
E_{n,H}{\left(q\right)}:=\frac{\widetilde{N}_{n,H}^{\xi}{\left(q\right)}}{\left(q-1\right)^{n}}.\label{eq:3.3.1}
\end{equation}
In accordance with \ref{rem:.3.2} and \ref{eq:3.1.7}, $\widetilde{N}_{n,H}^{\xi}{\left(q\right)}\in\mathbb{Z}[q]$.
Moreover, $\left(q-1\right)^{n}$ divides $\left|\mathrm{Sp}{\left(2n,\mathbb{F}_{q}\right)}\right|$
and combining Remark \ref{rem:20} and \ref{rem:21}, we get $\mathrm{gcd}{\left(\left(q-1\right),\chi{\left(1\right)}\right)}=1$
for all principal series $\chi$ of $\mathrm{Sp}{\left(2n,\mathbb{F}_{q}\right)}$
and this implies that $E_{n,H}{\left(q\right)}\in\mathbb{Z}[q]$.
\begin{thm}
\label{thm:main}
For all $\boldsymbol{Z}\leq H\leq\boldsymbol{\mu}_{\boldsymbol{2}}^{n}$, the variety
$\widetilde{\mathcal{M}}_{n,H}^{\xi}/\mathbb{C}$ has polynomial count
and its $E$-polynomial satisfies
\[
E{\left(\widetilde{\mathcal{M}}_{n,H}^{\xi}/\mathbb{C};q\right)}=E_{n,H}{\left(q\right)}.
\]
\end{thm}
\begin{proof}
From the definition \ref{eq:2.9} of $\widetilde{\mathcal{U}}_{n,H}^{\xi}$
it is clear that it can be viewed as a subscheme $\mathcal{X}_{H}$
of $\mathrm{Sp}\left(2n,R\right)^{2g}$ with $R:=\mathbb{Z}{\left[\zeta,\frac{1}{2m}\right]}$
and we can do the same thing for the $\widetilde{\mathcal{U}}_{i,H}$
as in Remark \ref{rem:2.6}, calling $\mathcal{X}_{i,H}$ the corresponding
subscheme over $R$ for $i\in I$. Let $\rho:R\rightarrow\mathbb{C}$
be an embedding, then $\mathcal{X}_{H}$ and $\mathcal{X}_{i,H}$
are spreading out of $\widetilde{\mathcal{U}}_{n,H}^{\xi}/\mathbb{C}$
and $\widetilde{\mathcal{U}}_{i,H}/\mathbb{C}$ respectively.

For every homomorphism
\begin{equation}
\phi:R\longrightarrow\mathbb{F}_{q}\label{eq:ext}
\end{equation}
the image $\phi{\left(\zeta\right)}$ is a primitive $2m$-root of unity
in $\mathbb{F}_{q}$, because the identity
\[
\prod_{i=1}^{2m-1}{\left(1-\zeta^{i}\right)}=2m
\]
guarantees that $1-\zeta^{i}$ is a unit in $R$ for $i=1,\ldots,2m-1$,
and therefore cannot be zero in the image. Hence all of our previous
considerations apply to compute $\left|\mathcal{X}_{H,\phi}{\left(\mathbb{F}_{q}\right)}\right|=\widetilde{N}_{n,H}^{\xi}{\left(q\right)}$.

On the other hand, the group scheme $T_{R}:=C_{\mathrm{Sp}{\left(2n,R\right)}}{\left(\xi\right)}$
acts on $\mathcal{X}_{H}$ and the $\mathcal{X}_{i,H}$'s by conjugation,
so using Seshadri's extension of geometric invariant theory quotients
for schemes (see \cite{key-14}), we can take the geometric quotient $\mathcal{Y}_{H}=\mathcal{X}_{H}/T_{R}$,
and we can define the affine scheme $\mathcal{Y}_{i,H}=\mathrm{Spec}{\left(R{\left[\mathcal{X}_{i,H}\right]}^{T_{R}}\right)}$
over $R$ for all $i\in I$. Then $\left\{ \mathcal{Y}_{i,H}\right\} _{i\in I}$
is an open cover of affine subschemes of $\mathcal{Y}_{H}$. Because
$\rho:R\rightarrow\mathbb{C}$ is a flat morphism, \cite[Lemma 2]{key-14} implies
that $\mathcal{Y}_{i,H}$ is a spreading out of $\widetilde{\mathcal{M}}_{i,H}$
as in Remark \ref{rem:2.9} for all $i$, so $\mathcal{Y}_{H}$ is a spreading
out of $\widetilde{\mathcal{M}}_{n,H}^{\xi}/\mathbb{C}$ because of
the local nature of fibered product for schemes.

Now take an $\mathbb{F}_{q}$- point of the scheme $\mathcal{Y}_{H,\phi}$
obtained from $\mathcal{Y}_{H}$ by the extension of scalars in \ref{eq:ext}.
By \cite[Lemma 3.2]{key-13}, the fiber over it in $\mathcal{X}_{H,\phi}{\left(\mathbb{F}_{q}\right)}$
is non empty and an orbit of ${\left(T/H\right)}{\left(\mathbb{F}_{q}\right)}$
and one can easily shows that ${\left(T/H\right)}{\left(\mathbb{F}_{q}\right)}$
acts freely on $\mathcal{X}_{H,\phi}{\left(\mathbb{F}_{q}\right)}$.
Consequently
\[
\left|\mathcal{Y}_{H,\phi}{\left(\mathbb{F}_{q}\right)}\right|=\frac{\left|\mathcal{X}_{H,\phi}{\left(\mathbb{F}_{q}\right)}\right|}{\left|{\left(T/H\right)}{\left(\mathbb{F}_{q}\right)}\right|}=\frac{\widetilde{N}_{n,H}^{\xi}{\left(q\right)}}{\left(q-1\right)^{n}}=E_{n,H}{\left(q\right)}.
\]
Thus $\widetilde{\mathcal{M}}_{n,H}^{\xi}/\mathbb{C}$ has polynomial
count. Now the theorem follows from Theorem \ref{thm:7}.
\end{proof}
Define 
\[
E_{n}{\left(q\right)}:=\frac{N_{n}^{\xi}{\left(q\right)}}{\left(q-1\right)^{n}}.
\]

\begin{cor}
\label{cor:Euler}
The $E$-polynomial of $\mathcal{M}_{n}^{\xi}/\mathbb{C}$ satisfies
\begin{equation}
E{\left(\mathcal{M}_{n}^{\xi}/\mathbb{C};q\right)}=E_{n}{\left(q\right)}=\frac{1}{\left(q-1\right)^{n}}\sum\limits_{\substack{\tau}}{\left(H_{\tau}{\left(q\right)}\right)^{2g-1}C_{\tau}}.\label{eq:3.3.3}
\end{equation}
\end{cor}
\begin{proof}
From Remark \ref{rem:strata}, we have that 
\begin{equation}
E{\left(\mathcal{M}_{n}^{\xi}/\mathbb{C};q\right)}=\sum\limits_{\substack{\boldsymbol{Z}\leq H\leq\boldsymbol{\mu}_{\boldsymbol{2}}^{n}}}{E{\left(\widetilde{\mathcal{M}}_{n,H}^{\xi}/\mathbb{C};q\right)}}\label{eq:3.3.4}
\end{equation}
so the corollary follows from Theorem \ref{thm:main} plugging \ref{eq:3.3.1}
into \ref{eq:3.3.4}.
\end{proof}
\begin{rem}
\label{rem:3.6}Combining Corollary \ref{cor:Euler} and Remark \ref{rem:.3.2},
we see that $E{\left(\mathcal{M}_{n}^{\xi}/\mathbb{C};q\right)}$ does
not depend on the particular choice of $\xi$ but only on the conditions
\ref{eq:2.1} that its eigenvalues have to satisfy. Thus we have actually
computed the $E$-polynomial of a very large family of parabolic character
varieties.
\end{rem}
\begin{example}
If we take $q\underset{8}{\equiv}1$ and divide both sides by $\left(q-1\right)$
in \ref{eq:1.1.4}, we get 
\[
\frac{N_{1}^{\xi}{\left(q\right)}}{q-1}={\left(q^{3}-q\right)^{2g-2}}{\left(q^{2}+q\right)}+{\left(q^{2}-1\right)^{2g-2}}{\left(q+1\right)}+{\left(2^{2g}-2\right)}{\left(q^{2}-q\right)^{2g-2}q}
\]
and by Corollary \ref{cor:Euler} and Remark \ref{rem:3.6}, this
is the $E$-polynomial of $\mathcal{M}_{1}^{\xi}/\mathbb{C}$ for
all possible choice of $\xi$. This result perfectly agrees with the
one obtained in \cite[Theorem 2]{key-15}.
\end{example}
We conclude this section proving the following interesting fact on $E{\left(\mathcal{M}_{n}^{\xi}/\mathbb{C};q\right)}$.
\begin{cor}
\label{cor:pal}
The $E$-polynomial of $\mathcal{M}_{n}^{\xi}/\mathbb{C}$ is palindromic.
\end{cor}
\begin{proof}
By Corollary \ref{cor:Euler}, it is sufficient to prove that $E_{n}{\left(q\right)}$ is palindromic. Since the degree of $E_{n}{\left(q\right)}$ is equal to $d_{n}={\left(2g-1\right)}n{\left(2n+1\right)}-n$ for Remark \ref{rem:28}, this is equivalent to prove that
\[
q^{d_{n}}E_{n}{\left(q^{-1}\right)}=E_{n}^{\xi}{\left(q\right)}.
\]
Now, because of \ref{eq:3.1.14}, $C_{\tau}=C_{\tau^{\prime}}$, where $\tau^{\prime}$ is the type dual to ${\tau}$ as in \ref{eq:1.10-1}, so the corollary follows using \ref{eq:1.13} in \ref{eq:3.3.3} after some little computations.
\end{proof}

\subsection{Connectedness and Euler characteristic}

We now deduce some important topological information on $\mathcal{M}_{n}^{\xi}/\mathbb{C}$. 
\begin{cor}
\label{cor:con}
The variety $\mathcal{M}_{n}^{\xi}/\mathbb{C}$ is connected.
\end{cor}
\begin{proof}
When $g=0$, $\mathcal{M}_{n}^{\xi}/\mathbb{C}$ is clearly empty,
unless $n=1$ when it is a point. Therefore, we can assume $g\geq1$
for the rest of the proof.

Proposition \ref{prop:2.13} tells us that $\mathcal{U}_{n}^{\xi}$ is smooth, while Corollary \ref{cor:2.14} says that each connected component of $\mathcal{M}_{n}^{\xi}$
has dimension $d_{n}$. Thus, by the last part of Remark \ref{rem:propr}, the leading coefficient of
$E{\left(\mathcal{M}_{n}^{\xi}/\mathbb{C};q\right)}$ is the number
of connected components of $\mathcal{M}_{n}^{\xi}$. By Corollary
\ref{cor:Euler}, $E{\left(\mathcal{M}_{n}^{\xi}/\mathbb{C};q\right)}=E_{n}{\left(q\right)}$,
so it is enough to determine the leading coefficient of $E_{n}^{\xi}{\left(q\right)}$.

By \ref{eq:3.3.3}, the top degree term in $E_{n}{\left(q\right)}$
corresponds to the biggest $H_{\tau}{\left(q\right)}$, that is attained
by the trivial character $1_{\mathrm{Sp}{\left(2n,\mathbb{F}_{q}\right)}}$
because of \ref{eq:1.12} and Remark \ref{rem:28}. Thus, using Remark \ref{rem:29}, the leading coefficient of
$E_{n}{\left(q\right)}$ is equal to $C_{\left(\hat{0},n,0,\beta\right)}$, where $\beta$ is the irreducible character of the Weyl group $W_{n}$ of $\mathrm{Sp}{\left(2n,\mathbb{F}_{q}\right)}$ corresponding to $1_{\mathrm{Sp}{\left(2n,\mathbb{F}_{q}\right)}}$.
Equation \ref{eq:3.1.14}, together with Remark \ref{rem:22}, tells us that $C_{\left(\hat{0},n,0,\beta\right)}=1$, so the corollary follows.
\end{proof}

\begin{cor}
\label{cor:char}
The Euler characteristic $\chi{\left(\mathcal{M}_{n}^{\xi}\right)}$
of $\mathcal{M}_{n}^{\xi}/\mathbb{C}$ vanishes for $g>1$. For
$g=1$, we have 
\[
\sum\limits_{\substack{n\geq0}}{\frac{\chi{\left(\mathcal{M}_{n}^{\xi}\right)}}{\left|W_{n}\right|}T^{n}}=\prod\limits_{\substack{k\geq1}}{\frac{1}{\left(1-T^{k}\right)^{3}}}=1+3T+9T^{2}+\cdots.
\]
\end{cor}
\begin{proof}
By Corollary \ref{cor:Euler} and 1 in Remark \ref{rem:3}, the Euler
characteristic of $\mathcal{M}_{n}^{\xi}/\mathbb{C}$ equals
\begin{equation}
E_{n}{\left(1\right)}=\sum\limits_{\substack{\tau}}{\left.\frac{\left(H_{\tau}{\left(q\right)}\right)^{2g-1}}{\left(q-1\right)^{n}}\right|_{q=1}C_{\tau}}.\label{eq:3.3.2}
\end{equation}
Since 
\begin{equation}
\left|\mathrm{Sp}{\left(2n,\mathbb{F}_{q}\right)}\right|={\left(q-1\right)^{n}}{\prod_{i=1}^{n}{{\left(q^{i}+1\right)}q^{2i-1}}}\label{eq:3.3.5}
\end{equation}
it follows from the definition \ref{eq:1.12} of $H_{\tau}{\left(q\right)}$
and Remark \ref{rem:21} that $\left(q-1\right)^{n\left(2g-2\right)}$
divides $\frac{\left(H_{\tau}{\left(q\right)}\right)^{2g-1}}{\left(q-1\right)^{n}}$,
so $E_{n}{\left(1\right)}=0$ when $g>1$ proving
the first assertion.

When $g=1$, plugging \ref{eq:3.3.5} in \ref{eq:1.12} and using Remark \ref{rem:21},
we get 
\begin{equation}
\left.\frac{H_{\tau}{\left(q\right)}}{\left(q-1\right)^{n}}\right|_{q=1}=\frac{\left|W_{n}\right|}{\left[W_{n}:S_{\lambda,\alpha_{1},\alpha_{\epsilon}}\right]\beta{\left(1\right)}}\label{eq:3.3.6}
\end{equation}
if $\tau=\left(\lambda,\alpha_{1},\alpha_{\epsilon},\beta\right)$
where $\lambda\vdash c$, $c+\alpha_{1}+\alpha_{\epsilon}=n$ and
$\beta\in\mathrm{Irr}{\left(S_{\lambda,\alpha_{1},\alpha_{\epsilon}}\right)}$,
so plugging \ref{eq:3.3.6} and \ref{eq:3.1.14} in \ref{eq:3.3.2}
for $g=1$ and summing over $\beta$, we have 
\begin{equation}
E_{n}{\left(1\right)}=\left|W_{n}\right|\sum\limits_{\substack{\lambda,\alpha_{1},\alpha_{\epsilon}}}{\frac{\left(-1\right)^{l{\left(\lambda\right)}}l{\left(\lambda\right)}!}{\prod\limits_{\substack{i}}{m_{i}{\left(\lambda\right)}!}}\left|\mathrm{Irr}{\left(S_{\lambda,\alpha_{1},\alpha_{\epsilon}}\right)}\right|}.\label{eq:3.3.2-1}
\end{equation}
Since $S_{\lambda,\alpha_{1},\alpha_{\epsilon}}={\left(\prod\limits_{\substack{i=1}}^{l{\left(\lambda\right)}}{S_{\lambda_{i}}}\right)}\times W_{\alpha_{1}}\times W_{\alpha_{\epsilon}}$,
it follows that 
\[
\mathrm{Irr}\left(S_{\lambda,\alpha_{1},\alpha_{\epsilon}}\right)=\prod\limits_{\substack{i=1}}^{{l{\left(\lambda\right)}}}{p{\left(\lambda_{i}\right)}\left|\mathrm{Irr}{\left(W_{\alpha_{1}}\right)}\right|\left|\mathrm{Irr}{\left(W_{\alpha_{\epsilon}}\right)}\right|}
\]
where $p{\left(\lambda_{i}\right)}$ is the number of partitions of
$\lambda_{i}$, for $i=1,\ldots,l\left(\lambda\right)$. Thus, if
we collect the partitions of the same size and length, summing over
$\alpha_{1}$ and $\alpha_{\epsilon}$ \ref{eq:3.3.2-1} becomes 
\begin{equation}
E_{n}{\left(1\right)}=\left|W_{n}\right|\sum\limits_{\substack{c=0}}^{n}{a_{c}b_{n-c}}\label{eq:3.3.7}
\end{equation}
with 
\[
a_{c}:=\sum\limits_{\substack{l\geq0}}{\sum\limits_{\substack{\lambda\vdash c\\l{\left(\lambda\right)}=l}}\frac{\left(-1\right)^{l}l!}{\prod\limits_{\substack{i}}{m_{i}{\left(\lambda\right)}!}}\prod\limits_{\substack{i=1}}^{l}{p{\left(\lambda_{i}\right)}}}
\]
and
\[
b_{n-c}:=\sum\limits_{\substack{\alpha_{1},\alpha_{\epsilon}\geq 0\\ \alpha_{1}+\alpha_{\epsilon}=n-c}}{\left|\mathrm{Irr}{\left(W_{\alpha_{1}}\right)}\right|\left|\mathrm{Irr}{\left(W_{\alpha_{\epsilon}}\right)}\right|}
\]
so
\begin{equation}
\sum\limits_{\substack{n\geq0}}{\frac{E_{n}{\left(1\right)}}{\left|W_{n}\right|}T^{n}}={\left(\sum\limits_{\substack{c\geq0}}{a_{c}T^{c}}\right)}{\left(\sum\limits_{\substack{m\geq0}}{b_{m}T^{m}}\right)}\label{eq:3.3.10-1}
\end{equation}

Now, it is easy to see that 
\[
\sum\limits_{\substack{c\geq0}}{a_{c}T^{c}}=\sum\limits_{\substack{l\geq0}}{\left(-\sum\limits_{\substack{n\geq1}}{p{\left(n\right)}T^{n}}\right)^{l}=\frac{1}{\sum\limits_{\substack{n\geq0}}{p{\left(n\right)}T^{n}}}}
\]
so by an identity of Euler, we get 
\begin{equation}
\sum\limits_{\substack{c\geq0}}{a_{c}T^{c}}=\prod\limits_{\substack{k\geq1}}{\left(1-T^{k}\right)}.\label{eq:3.3.8}
\end{equation}

On the other hand, it is known that
\[
\sum\limits_{\substack{n\geq0}}{\left|\mathrm{Irr}{\left(W_{n}\right)}\right|T^{n}}=\prod\limits_{\substack{k\geq1}}{\frac{1}{\left(1-T^{k}\right)^{2}}}
\]
hence
\begin{equation}
\sum\limits_{\substack{m\geq0}}{b_{m}T^{m}}=\prod\limits_{\substack{k\geq1}}{\frac{1}{\left(1-T^{k}\right)^{4}}}.\label{eq:3.3.9}
\end{equation}

Thus the second assertion of the corollary follows by plugging \ref{eq:3.3.8}
and \ref{eq:3.3.9} in \ref{eq:3.3.10-1}.
\end{proof}

\end{document}